\documentclass{amsart}

\usepackage{amscd,amsmath,amssymb,amsfonts,bbm, calligra}
\usepackage[T1]{fontenc}
\usepackage{lmodern}
\usepackage{mathtools}
\usepackage{enumerate}
\usepackage{multicol}
\usepackage[all]{xy}
\usepackage{mathrsfs}
\usepackage{footmisc}
\usepackage[unicode,pdfborder={0 0 0},final]{hyperref}

\newtheorem{thm}[equation]{Theorem}
\newtheorem{prop}[equation]{Proposition}

\newtheorem{lem}[equation]{Lemma}
\newtheorem{cor}[equation]{Corollary}

\theoremstyle{definition}

\theoremstyle{remark}
\newtheorem{rem}[equation]{Remark}
\newtheorem{rems}[equation]{Remarks}
\newtheorem{ex}[equation]{Example}
\newcounter{counterprop}

\newtheorem*{property*}{Property}

\numberwithin{equation}{section}

\newcommand{\Id}{\mathrm{Id}}

\newcommand{\End}{\mathrm{End}}
\newcommand{\Ker}{\mathrm{Ker}}
\newcommand{\Ima}{\mathrm{Im}}
\newcommand{\cl}{\mathrm{cl}}

\newcommand{\homo}{\mathrm{hom}}

\newcommand{\Spec}{\mathrm{Spec}}

\newcommand{\Hdg}{\mathrm{Hdg}}
\newcommand{\Hod}{\mathrm{Hod}}
\newcommand{\Gal}{\mathrm{Gal}}
\newcommand{\Sp}{\mathrm{Sp}}
\newcommand{\Res}{\mathrm{Res}}

\def\alg{\mathrm{alg}}

\newcommand{\isoto}{\myxrightarrow{\,\sim\,}}

\makeatletter
\def\myrightarrow{{\setbox\z@\hbox{$\rightarrow$}\dimen0\ht\z@\multiply\dimen0 6\divide\dimen0 10\ht\z@\dimen0\box\z@}}
\def\myrightarrowfill@{\arrowfill@\relbar\relbar\myrightarrow}
\newcommand{\myxrightarrow}[2][]{\ext@arrow 0359\myrightarrowfill@{#1}{#2}}
\makeatother

\makeatletter
\newcommand{\extp}{\@ifnextchar^\@extp{\@extp^{\,}}}
\def\@extp^#1{\mathop{\bigwedge\nolimits^{\!#1}}}
\makeatother

\newcommand{\CH}{\mathrm{CH}}

\def\cM{\mathcal{M}}
\def\cL{\mathcal{L}}
\def\cO{\mathcal{O}}

\def\cF{\mathcal{F}}
\def\cU{\mathcal{U}}
\def\cV{\mathcal{V}}
\def\cI{\mathcal{I}}
\def\ci{\mathcal{C}^{\infty}}

\def\kW{\mathfrak{W}}

\def\kZ{\mathfrak{Z}}

\def\uG{\underline{G}}
\def\uF{\underline{F}}
\def\ul{\underline{l}}

\def\oZ{\overline{Z}}

\def\ow{\overline{w}}

\def\oD{\overline{D}}

\def\wkW{\widetilde{\mathfrak{W}}}

\def\wW{\widetilde{W}}

\def\wf{\widetilde{f}}

\def\D{\mathbb D}
\def\Z{\mathbb Z}
\def\C{\mathbb C}

\def\Q{\mathbb Q}
\def\P{\mathbb P}
\def\R{\mathbb R}

\begin{document}

\title[On the subvarieties with nonsingular real loci]{On the subvarieties with nonsingular real loci of a real algebraic variety}

\author{Olivier Benoist}
\address{D\'epartement de math\'ematiques et applications, \'Ecole normale sup\'erieure, CNRS,
45 rue d'Ulm, 75230 Paris Cedex 05, France}
\email{olivier.benoist@ens.fr}

\renewcommand{\abstractname}{Abstract}
\begin{abstract}
Let $X$ be a smooth projective real algebraic variety.
We give new positive and negative results on the problem of approximating a submanifold of the real locus of 
$X$ by real loci of subvarieties of $X$, as well as on the problem of determining the subgroups of the Chow groups of $X$ generated by subvarieties with nonsingular real loci, or with empty real loci. 
\end{abstract}
\maketitle

\section*{Introduction}
\label{intro}

In this article, we study the subvarieties with nonsingular real loci of a smooth projective real algebraic variety $X$. We consider their classes in the Chow groups of~$X$ and whether their real loci can approximate a fixed $\ci$ submanifold of $X(\R)$.

Let $c$ and $d$ denote the codimension and the dimension of these subvarieties.
The guiding principle of our results is that, for each of the three problems that we will consider in \S\S\ref{Chow}-\ref{approximation}, subvarieties with nonsingular real loci are abundant when $d<c$ (see Theorems~\ref{Chow1}, \ref{Chow2} and \ref{thC1}), but may be scarce for $d\geq c$ (see Theorems \ref{Chow3}, \ref{Chow4} and~\ref{thC2}). The geometric rationale behind this principle, in the spirit of Whitney's theorem in  differential geometry \cite{Whitneydiff}, is that a $d$\nobreakdash-dimensional variety mapped generically to $X$ is expected not to self-intersect, hence to have nonsingular image in $X$, only if~$d< c$.

\subsection{Chow groups}
\label{Chow}

It is an old question, going back to Borel and Haefliger \cite[\S 5.17]{BH},
to decide when the Chow group $\CH_d(X)$ of a smooth projective variety $X$ of dimension $c+d$ over a field is generated by classes of smooth subvarieties of $X$. 
This is not true in general (a first counterexample appeared in \cite[Theorem~1]{HRT}, for $c=2$ and $d=7$). The main positive result, due to Hironaka \cite[Theorem~p.~50]{Hironakasmoothing}, gives an affirmative answer if $d<c$ and $d\leq 3$ (his arguments now work over any infinite perfect field, thanks to \cite{CP}). 
One may wonder whether Hironaka's theorem holds as soon as~$d<c$.


In real algebraic geometry, it is natural to consider, more generally, subvarieties that are smooth along their real loci. Our first theorem is a variant of Hironaka's result in this setting, valid for all values of $(c,d)$ such that $d<c$.

 \begin{thm}[Theorem \ref{Chowth}]
 \label{Chow1}
 Let $X$ be a smooth projective variety of dimension $c+d$ over $\R$. 
If $d<c$, then the group $\CH_d(X)$ is generated by classes of closed subvarieties of~$X$ that are smooth along their real loci.
  \end{thm}

 Our proof is based on the smoothing technique developed by Hironaka in \cite{Hironakasmoothing}. We need to refine it for two reasons: to control real loci, and to deal with the singularities that inevitably appear in the course of our proof if $d>3$. To do so, we rely on the theory of linkage, as developed by Peskine and Szpiro \cite{PS} and Huneke and Ulrich \cite{HUDCG}. Our argument works over an arbitrary real closed field.
 
\vspace{.5em}

Our second theorem shows that Theorem \ref{Chow1} is optimal, for infinitely many values of~$c$.
We let $\alpha(m)$ denote the number of ones in the dyadic expansion of $m$.
  
 \begin{thm}[Theorem \ref{Chowth3}]
 \label{Chow3}
If $d\geq c$ are such that $\alpha(c+1)\geq 3$, there exists an abelian variety $X$ of dimension $c+d$ over $\R$ such that $\CH_d(X)$ is not generated by classes of closed subvarieties of $X$ that are smooth along their real loci.
 \end{thm} 
 Theorem \ref{Chow3} is entirely new. The hypothesis that $\alpha(c+1)\geq 3$ cannot be weakened to ${\alpha(c+1)\geq 2}$.
 Indeed, Kleiman has showed that the Chow group of codimension~$2$ cycles on a smooth projective fourfold or fivefold over an infinite field is generated by classes of smooth subvarieties (see \cite[Theorem 5.8]{KleimanGr}, 
where the hypothesis that the base field is algebraically closed may be discarded as a theory of Chow groups and Chern classes is now available in the required generality \cite{Fulton}). 

  Let us briefly explain the principle of the proof of Theorem \ref{Chow3} in the key case where $c=d$. Assume to simplify that $\beta\in\CH_d(X)$ is the class of closed subvariety $Y\subset X$ which is smooth along $Y(\R)$. Let $g:W\to X$ be a morphism obtained by resolving the singularities of~$Y$. The double locus of $g$, which is well-defined as a $0$-cycle on~$W$, has degree divisible by~$4$. Indeed, double points come two by two, and each such pair has a distinct complex conjugate. On the other hand, a double point formula due to Fulton \cite[\S 9.3]{Fulton} computes the degree of this double locus in terms of the Chern classes of $X$ and~$W$ and of the self-intersection of $Y$ in $X$. Divisibility results for Chern numbers due to Rees and Thomas \cite[Theorem 3]{RT} now give restrictions on $\beta$, which sometimes lead to a contradiction. 

 This strategy applies as well over $\C$, and yields new examples of smooth projective complex varieties whose Chow groups are not generated by  smooth subvarieties.
 
  \begin{thm}[Theorem \ref{Chowth5}]
  \label{Chow5}
If $d\geq c$ are such that $\alpha(c+1)\geq 3$, there exists a smooth projective variety~$X$ of dimension $c+d$ over $\C$ such that $\CH_d(X)$  is not generated by classes of smooth closed subvarieties of $X$.
\end{thm}

 This complements the counterexamples of \cite[Theorem~1]{HRT}, \cite[Th\'eo\-r\`eme~6]{Debarre} and \cite[Theorem 1.2]{BD} to the question of Borel and Haefliger. Theorem~\ref{Chow5} is closely related to \cite[Proposition 1]{RTac}, where the easier problem of showing that a cycle is not rationally equivalent to the class of a smooth subvariety (as opposed to a linear combination of classes of smooth subvarieties) is considered.

\subsection{The kernel of the Borel--Haefliger map}
\label{BH}
    
If $X$ is a smooth projective variety of dimension $c+d$ over $\R$, a related problem is to determine the subgroup of $\CH_d(X)$ generated by classes of subvarieties with no real points. 
This subgroup is included in the kernel $\CH_d(X)_{\R-\homo}$ of the Borel--Haefliger cycle class map ${\cl_{\R}:\CH_d(X)\to H_d(X(\R),\Z/2)}$ (see \cite{BH} or \cite[\S 1.6.2]{BW1}), which associates with the class of an integral closed subvariety $Y\subset X$ the homology class of its real locus. One may wonder when these two subgroups coincide. This question was known to have a positive answer for $c=1$ (Br\" ocker \cite{Brocker}, see also \hbox{\cite[\S 4]{Scheidererpurity}}), for $d=0$ (Colliot-Th\'el\`ene and Ischebeck \cite[Proposition~3.2~(ii)]{CTI}), and for $d=1$ and $c=2$ (Kucharz \cite[Theorem~1.2]{KucChow}). 

Combining our improvements of Hironaka's smoothing technique and a theorem of Ischebeck and Sch\"ulting according to which $\CH_d(X)_{\R-\homo}$ is generated by classes of integral closed subvarieties of $X$ whose real locus is not Zariski-dense \cite[Main Theorem 4.3]{IS}, we obtain an affirmative answer for all values of $(c,d)$ such that $d<c$.
  
   \begin{thm}[Theorem \ref{Chowth2}]
 \label{Chow2}
 Let $X$ be a smooth projective variety of dimension $c+d$ over $\R$. 
 If $d<c$, then $\CH_d(X)_{\R-\homo}$ is generated by classes of closed integral subvarieties of~$X$ with empty real loci.
  \end{thm}
  
  Theorem \ref{Chow2} fails in general over non-archimedean real closed fields (see Remark~\ref{remrealclosed}). Kucharz has shown in \cite[Theorem 1.1]{KucChow} that the hypothesis $d<c$ of Theorem~\ref{Chow2} cannot be improved, for all even values of $c\geq 2$.
  We extend this result to all the values of $c$ not of the form $2^k-1$.
  
  \begin{thm}[Theorem \ref{Chowth4}]
 \label{Chow4}
If $d\geq c$ are such that $\alpha(c+1)\geq 2$, there exists an abelian variety $X$ of dimension $c+d$ over $\R$ such that $\CH_d(X)_{\R-\homo}$ is not generated by classes of closed subvarieties of $X$ with empty real loci.
\end{thm}

  The proof of Theorem \ref{Chow4} follows the same path as that of Theorem \ref{Chow3}, using additionally a new result on congruences of Chern numbers (Theorem \ref{divtop}).
The hypothesis that $\alpha(c+1)\geq 2$ in Theorem \ref{Chow4} cannot be removed in general, in view of Br\"ocker's above-mentioned theorem \cite{Brocker} when $c=1$.

\subsection{Algebraic approximation} 
\label{approximation}

Let $X$ be a smooth projective variety of dimension $c+d$ over~$\R$, and fix a closed $d$-dimensional $\ci$ submanifold $j:M\hookrightarrow X(\R)$.
We now focus on the classical question whether $M$ can be approximated by real loci of algebraic subvarieties of $X$ (see \cite[Definition 12.4.10]{BCR} or \cite[\S 2.8]{AKbook}):

\begin{property*}[A]
For all neighbourhoods $\mathcal{U}\subset\ci(M,X(\R))$ of the inclusion, there exist ${\phi\in\cU}$ and a closed subvariety $Y\subset X$ which is smooth along $Y(\R)$ such that $\phi(M)=Y(\R)$.
\end{property*}

One must take into account the topological obstructions to the validity of $(A)$. The finest ones are based on cobordism theory and originate from \cite{BTvb} or \cite{AKIHES}. If~$T$ is a topological space, recall that two continuous maps $f_1:N_1\to T$ and ${f_2:N_2\to T}$, where the $N_i$ are $d$-dimensional compact~$\ci$ manifolds, are said to be \textit{cobordant} if there exist a compact $\ci$ manifold with boundary $C$, a diffeomorphism ${\partial C\simeq N_1\cup N_2}$, and a continuous map $f:C\to T$ such that $f|_{N_i}=f_i$ for $i\in\{1,2\}$. The group (for the disjoint union) of cobordism equivalence classes of such maps is the $d$-th \textit{unoriented cobordism group} $MO_d(T)$ of $T$.

Let $MO_d^{\alg}(X(\R))\subset MO_d(X(\R))$ be the subgroup generated by cobordism classes of continuous maps of the form $g(\R):W(\R)\to X(\R)$, where $g:W\to X$ is a morphism of smooth projective varieties over $\R$ and $W$ has dimension $d$.  The following property is a necessary condition for the validity of $(A)$.

\begin{property*}[B]
One has $[j:M\hookrightarrow X(\R)]\in MO_d^{\alg}(X(\R))$.
\end{property*}

We show that it is the only obstruction to the validity of $(A)$ for low values of~$d$.

\begin{thm}[Theorem \ref{approxth}]
\label{thC1}
Properties $(A)$ and $(B)$ are equivalent if $d<c$.
\end{thm}
 
 Theorem \ref{thC1} was already known when $d=1$, thanks to Bochnak and Kucharz \cite[Theorem~1.1]{BKsub} for $c=2$, and to Wittenberg and the author \cite[Theorem~6.8]{BW2} for any $c\geq 2$ (improving earlier results by Akbulut and King \cite{AKcourbes}). 
Theorem~\ref{thC1} is new for $d\geq 2$. 
 
  Our proof is based on a relative Nash--Tognoli theorem (see \cite[Proposition~4.1]{BTvb} or \cite[Proposition 0.2]{AKIHES}),
which solves the approximation problem up to unwanted singular points. To remove these singular points, we use the refinements of Hironaka's smoothing method already mentioned in \S\S\ref{Chow}-\ref{BH}. 
Hironaka's smoothing technique, as developed in \cite{Hironakasmoothing}, had already been applied in the context of real algebraic approximation in the proof of \cite[Theorem~6.8]{BW2}.
  
We also prove that Theorem \ref{thC1} is sharp: it may fail as soon as $d\geq c$, for infinitely many values of $c$.
Recall that $\alpha(m)$ is the number of ones in the dyadic expansion of $m\geq 0$.
 
 \begin{thm}[Theorem \ref{thji}]
 \label{thC2}
If $d\geq c$ are such that $\alpha(c+1)=2$,  there exist $X$ and $M$ such that $(A)$ fails but $(B)$ holds.
 \end{thm}
To the best of our knowledge, Theorem \ref{thC2} features the first examples demonstrating that properties~$(A)$ and $(B)$ are not equivalent in general. 
The hypothesis that $\alpha(c+1)=2$ cannot be entirely dispensed with, as $(A)$ and~$(B)$ are equivalent when $c=1$ (see Proposition~\ref{hyp}).
 The values of~$c$ for which Theorem \ref{thC2} applies are $c\in\{2,4,5,8,9,11,16,\dots\}$. 
 We have not been able to disprove the equivalence of $(A)$ and $(B)$ for other values of $c$, for instance for $c=3$. 

 The proof of Theorem \ref{thC2} uses techniques similar to that of Theorems \ref{Chow3} and~\ref{Chow4}.
Although both properties $(A)$ and $(B)$ only involve real loci, the proof of Theorem~\ref{thC2} makes use in an essential way of global topological properties of sets of complex points, through their classes in the complex cobordism ring $MU_*$. 

In \cite[p.~269]{KvH}, Kucharz and van Hamel ask whether 
property $(A)$ always holds when $X=\P^n_{\R}$.
The obstructions used in the proof of Theorem~\ref{thC2} show that this question would have a negative answer if one replaced $\P^n_{\R}$ with other very similar varieties, such as some products of projective spaces (property~$(B)$ always holds for these varieties by \cite[Remark 3 p.~103]{BTvb}). This demonstrates the very particular role played by projective spaces in the question of Kucharz and van~Hamel.
 
 \begin{thm}[Theorem \ref{projth}]
 \label{thP}
For all $k\geq 1$, property $(A)$ fails in general for $c=d=2^k$ and $X=\P^1_{\R}\times\P^{2^{k+1}-1}_{\R}$.
 \end{thm}

\subsection{Structure of the article}
 
We study linkage to expand the scope of Hironaka's smoothing technique in \S\ref{linkage}, and use it in \S\ref{sectionapprox} to prove Theorems~\ref{Chow1}, \ref{Chow2} and~\ref{thC1}. Generalities about complex cobordism and an application to the divisibility of the top Segre class may be found in \S\ref{complexcobordism}. This result and a double point formula are combined in \S\ref{doublepoint} to prove Theorems \ref{Chow3}, \ref{Chow5}, \ref{Chow4}, \ref{thC2} and \ref{thP}. Finally, variants of properties $(A)$ and $(B)$ and their interactions are considered in \S\ref{final}.

\subsection{Acknowledgements}
 
I thank Michele Ancona and the referees for useful comments that improved the exposition of the article.

 \subsection{Notation and conventions}
 \label{notation}
 
 A variety over a field $k$ is a separated scheme of finite type over $k$. Smooth varieties over $k$ are understood to be equidimensional. If $f:X\to Y$ is a morphism of varieties over $k$ and $k'$ is a field extension of $k$, we let $f(k'):X(k')\to Y(k')$ be the map induced at the level of $k'$-points. We denote by $\R$ and $\C$ the fields of real and complex numbers. 
 
 All $\ci$ manifolds are assumed to be Hausdorff and second countable. We endow the set $\ci(M,N)$ of $\ci$ maps between two $\ci$ manifolds with the weak $\ci$ topology \cite[p.~36]{Hirsch}.

For $m\geq 0$, we let $\alpha(m)$ be the number of ones in the dyadic expansion of $m$.

 \section{Linkage}
 \label{linkage}
 
In the whole of \S\ref{linkage}, we fix an infinite field $k$, a smooth projective variety $V$ over~$k$, a very ample line bundle $\cO_V(1)$ on $V$, and a (possibly empty) Cohen--Macaulay closed subscheme $W\subset V$ of pure codimension $r$ in $V$. 

We study the subvarieties of $V$ that are linked to $W$ by complete intersections defined by sections of multiples of~$\cO_V(1)$~(\S\ref{link}), and their behaviour in families (\S\S\ref{family}-\ref{generic}), focusing in particular on their images by a morphism and on their real loci when $k$ is a real closed field (\S\S\ref{morphism}-\ref{realpar}).

If $\ul=(l_1,\dots,l_r)$ is an $r$-tuple of integers and if $\cF$ is a coherent sheaf on $V$, we set $\cF(\ul):=\bigoplus_{i=1}^r\cF(l_i)$. In particular, $H^0(V,\cF(\ul))=\bigoplus_{i=1}^r H^0(V,\cF(l_i))$.
 A statement depending on an $r$-tuple of integers $\ul=(l_1,\dots,l_r)$ is said to hold for $\ul\gg0$ if it holds for $l_r\gg\dots\gg l_2\gg l_1\gg 0$, i.e., if $l_1$ is big enough, if $l_2$ is big enough (depending on $l_1$), and so forth.

 \subsection{Linked subvarieties} 
 \label{link}
 
Let $\cI_W\subset\cO_V$ be the ideal sheaf of $W$ in $V$. Choose an $r$-tuple of integers $\underline{l}=(l_1,\dots,l_r)$ and a section $\uF\in H^0(V,\cI_W(\ul))$ such that $\underline{F}=(F_1,\dots,F_r)$ is a regular sequence (such $\uF$ always exists if $\cI_W(\ul)$ is generated by its global sections, for instance for $\ul\gg0$).
 
  Let ${Z:=\{F_1=\dots=F_r=0\}}\subset V$ be the complete intersection it defines, and let $\cI_Z=\langle\uF\rangle\subset\cO_V$ be its ideal sheaf. 
 Let $W'\subset V$ be the subvariety with ideal sheaf $\cI_{W'}:=(\cI_Z:\cI_W)\subset\cO_V$, where a local section $s\in \cO_V$ belongs to $(\cI_Z:\cI_W)$ if multiplication by $s$ induces a morphism $\cI_W\xrightarrow{s}\cI_Z$. 
  
 One has $Z=W\cup W'$ set-theoretically. It is a theorem of Peskine and Szpiro \cite[Proposition 1.3]{PS} that $W'\subset V$ is also Cohen--Macaulay of pure codimension~$r$, and that $\cI_{W}=(\cI_Z:\cI_{W'})\subset\cO_V$. In view of the symmetry of the relation between the subschemes $W$ and $W'$ of $V$, they are said to be \textit{linked} by the regular sequence~$\uF$. We write $W\sim W'$, or $W\sim_{\ul} W'$ if we want to emphasize that the regular sequence is a section of $\cO_V(\ul)$. We also say that $W\sim W'$ is the \textit{link} defined by $\uF$.
 
 \begin{rems}
(i) In the whole of \S\ref{linkage}, we could have only considered links with respect to complete intersections of multidegree $(l,\dots,l)$, with $l\gg 0$ when needed. The reason why we allow multidegrees $\ul=(l_1,\dots,l_r)$, requiring $\ul\gg 0$ when needed, is to be able to apply directly the proof of \cite[Lemma 5.1.1]{Hironakasmoothing} in the proof of Proposition \ref{Hironaka} below.
 
(ii)  In \S\S\ref{link}--\ref{generic}, we could allow $V$ to be any Gorenstein projective variety (see especially \cite[Proposition 1.3]{PS}).
\end{rems}
 
  \begin{lem}
  \label{choosesequence}
 Let $x\in V$, and let $g_1,\dots,g_r\in\cI_{W,x}\subset \cO_{V,x}$ be a regular sequence. Then, for an $r$-tuple of integers $\ul\gg 0$, there exists a regular sequence $\uF\in H^0(V,\cI_W(\ul))$ such that the ideals $\langle \uF\rangle$ and $\langle g_1,\dots,g_r\rangle$ of $\cO_{V,x}$ coincide.
 \end{lem}
 
 \begin{proof}
Let $Y\subset V$ and $Y_i\subset V$ be the schematic closures of $\{g_1=\dots=g_r=0\}$ and $\{g_i=0\}$ in $V$. Let $\cI_W,\cI_Y,\cI_{Y_i}\subset\cO_V$ be the ideal sheaves of $W$, $Y$ and $Y_i$ in $V$ and define $\cI:=\cI_W\cap\cI_Y$ and $\cI_i:=\cI_W\cap\cI_{Y_i}$. The subscheme of $V$ defined by $\cI$ has support $Y\cup W$, hence has pure codimension $r$ in $V$, and coincides with $Y$ in a neighbourhood of $x$. 
 
Choose $\ul\gg 0$ so that the sheaves $\cI(\ul)$ and $\cI_i(l_i)$ are generated by their global sections, and choose a general element $\uF\in H^0(V,\cI(\ul))$. Since $\cI_i(l_i)$ is globally generated, there exist $G_i\in H^0(V,\cI_i(l_i))\subset H^0(V,\cI(l_i))$ with $\langle G_i\rangle=\langle g_i\rangle\subset\cO_{V,x}$, hence with $\langle G_1,\dots,G_r\rangle=\langle g_1,\dots,g_r\rangle\subset\cO_{V,x}$. Since $\uF$ has been chosen general, we deduce that $\langle \uF\rangle=\langle g_1,\dots,g_r\rangle\subset\cO_{V,x}$. Since $\cI(\ul)$ is globally generated, we also see that $\uF$ forms a regular sequence. The lemma is proven.
 \end{proof}
 
 \begin{lem}
 \label{disappear}
 Let $W=W_0\sim_{\ul_1} W_1\sim_{\ul_2}\dots\sim_{\ul_j}W_j$ be links of subschemes of~$V$. Assume that $W$ is a local complete intersection at $x\in W$. For all $r$-tuples of integers $\ul_{2j+1}\gg\dots\gg \ul_{j+1}\gg0$, there exists a chain
 $W_j\sim_{\ul_{j+1}} W_{j+1}\sim_{\ul_{j+2}}\dots\sim_{\ul_{2j+1}}W_{2j+1}$ of links of subschemes of $V$ such that $x\notin W_{2j+1}$.
 \end{lem}
 
 \begin{proof}
For $1\leq i\leq j$, let  $\underline{F}_i$ be the regular sequence yielding the link $W_{i-1}\sim_{\ul_i} W_i$. 
Thanks to Lemma \ref{choosesequence}, one may choose inductively, for $j+1\leq i\leq 2j$, an $r$-tuple $\ul_i\gg 0$, a regular sequence $\uF_i\in H^0(V,\cI_{W_{i-1}}(\ul_i))$ such that the ideals $\langle \underline{F}_i\rangle$ and $\langle \underline{F}_{2j+1-i}\rangle$ of $\cO_{V,x}$ coincide. This gives rise to a link $W_{i-1}\sim_{\ul_i} W_i$ with the property that $W_i$ and $W_{2j-i}$ coincide in a neighourhood of $x$, by the symmetry of the link construction.

 The subschemes $W_{2j}$ and $W_0=W$ then coincide in a neighbourhood of $x$, hence $W_{2j}$ is a local complete intersection at $x$, defined by a regular sequence $g_1,\dots,g_r\in\cI_{W_{2j},x}$. A final application of Lemma \ref{choosesequence} provides us with an $r$-tuple $\ul_{2j+1}\gg0$, and with a link $W_{2j}\sim_{\ul_{2j+1}}W_{2j+1}$ associated with a regular sequence $\underline{F}_{2j+1}$ such that $\langle \underline{F}_{2j+1}\rangle=\langle g_1,\dots,g_r\rangle\subset \cO_{V,x}$. It follows that $x\notin W_{2j+1}$.
 \end{proof}

 \subsection{Linkage in families}
 \label{family}
 
Let $B$ be a smooth variety over $k$. Let $\kW\subset V\times B$ be a closed subscheme of pure codimension $r$ with ideal sheaf $\cI_{\kW}\subset\cO_{V\times B}$, such that the second projection $f:\kW\to B$ is flat with Cohen--Macaulay fibers. Let $\ul$ be an $r$\nobreakdash-tuple such that, letting $p:V\times B\to B$ denote the second projection, the adjunction morphism $p^*p_*(\cI_{\kW}(\ul))\to\cI_{\kW}(\ul)$ is surjective and $R^ip_*(\cI_{\kW}(\ul))=0$ for $i>0$. Note that, under these conditions, the push-forward sheaf $E:=p_*(\cI_{\kW}(\ul))$ is a vector bundle such that the natural morphism $E|_b\to H^0(V,\cI_{\kW_b}(\ul))$ is an isomorphism for all $b\in B$ \cite[III,~Theorem~12.11]{Hartshorne}.
  
View $E\to B$ as a geometric vector bundle over $B$. A point of $E$ over $b\in B$ corresponds to a section $\uF\in H^0(V,\cI_{\kW_b}(\ul))$. Let $B'\subset E$ be the open subset of those points such that $\uF$ forms a regular sequence, and hence defines a complete intersection in $V$. Let $\kZ\subset V\times B'$ be the universal family of these complete intersections, and let $\kW_{B'}\subset V\times B'$ be the base change of $\kW$, with ideal sheaves $\cI_{\kZ},\cI_{\kW_{B'}}\subset\cO_{V\times B'}$. We consider the subscheme $\kW'\subset V\times B'$ with ideal sheaf $\cI_{\kW'}:=(\cI_{\kZ}:\cI_{\kW_{B'}})$ and we let $f':\kW'\to B'$ denote the second projection.

 By Proposition \ref{linkrel}, this extends the construction of \S\ref{link} in the relative setting.
  
 \begin{prop}
 \label{linkrel}
 The morphism $f':\kW'\to B'$ is flat with Cohen--Macaulay fibers.
 
 For $b\in B'$, one has $\cI_{\kW'_b}=(\cI_{\kZ_b}:\cI_{\kW_b})$.
 \end{prop}
  
  \begin{proof}
  The scheme $\kW$ is Cohen--Macaulay by \cite[Corollaire 6.3.5 (ii)]{EGA42}, hence so is $\kW'$ by \cite[Proposition 1.3]{PS}. Since $f'$ is equidimensional with regular base and Cohen--Macaulay total space, it is flat by \cite[Proposition 6.1.5]{EGA42}.
 
Choose a regular system of parameters $x_1,\dots,x_N$  of the regular local ring $\cO_{B',b}$. To show the equality of ideals $\cI_{\kW'_b}=(\cI_{\kZ_b}:\cI_{\kW_b})$ at a point $v\in  V_b$, one may apply $N$ times sucessively \cite[Lemma 2.12]{HUDCG} in the local ring $\cO_{V_b,v}$ (this is essentially what is done in \cite[Proposition 2.13]{HUDCG}).
 That the fibers $\kW'_b$ of $f'$ are Cohen--Macaulay now follows from \cite[Proposition~1.3]{PS}.
  \end{proof}

 \subsection{Moduli of links}
 \label{generic}
 
We now iterate the construction of \S\ref{family}, thus adapting to our global setting a local construction due to Huneke and Ulrich \cite{HUstructure, HUalgebraic}.
  
Recall that $W\subset V$ is a Cohen--Macaulay closed subscheme of pure codimension~$r$.  We set $L_0(W):=\Spec(k)$, $\kW_0:=W$ and $f_0:\kW_0\to L_0(W)$ to be the structural morphism. We inductively construct $f_i:\kW_i\to L_i(W)$ for $i\geq 1$, by choosing an $r$-tuple $\ul_i\gg 0$, by applying the construction of \S\ref{family} to $f_{i-1}:\kW_{i-1}\to L_{i-1}(W)$, and by setting $(f_i:\kW_i\to L_i(W))=(f'_{i-1}:\kW'_{i-1}\to L_{i-1}(W)')$. The varieties $L_i(W)$ are irreducible, smooth over $k$ and $k$-rational, and the morphisms $f_i$ are flat with Cohen--Macaulay fibers, as an induction based on Proposition \ref{linkrel} shows.
  
We call $f_j:\kW_j\to L_j(W)$ the $j$-th \textit{moduli of links} 
 of $W$ (with respect to the degrees $\ul_1,\dots, \ul_j$).
 Its points parametrize sequences $(\underline{F}_i)_{1\leq i\leq j}$ of regular sequences that give rise to chains of linked subvarieties $W=W_0\sim_{\ul_1} W_1\sim_{\ul_2}\dots\sim_{\ul_j}W_j$ of $V$.

\begin{rem}
The construction of the $j$-th moduli of links $f_j:\kW_j\to L_j(W)$ goes through with no modifications if $k$ is finite, but beware that $L_j(W)(k)$ might be empty in this case.
\end{rem}

\subsection{Images by a morphism}   
\label{morphism}
  
In \S\S\ref{morphism}-\ref{realpar}, we fix a smooth morphism ${\pi:V\to X}$ of smooth projective varieties over $k$
and we let $d$ and $n$ be the dimensions of~$W$ and~$X$. 
In Propositions \ref{Hironaka}, \ref{propBertini}, \ref{propinj} and \ref{Rlink},
 we study the images by~$\pi$ of subvarieties of~$X$ linked to $W$, under the assumption that $n>2d$.  
Proposition~\ref{Hironaka} is due to Hironaka \cite{Hironakasmoothing}. Proposition~\ref{propBertini} is a simple Bertini theorem. Proposition~\ref{propinj}, which is the main result of \S\ref{morphism}, is more delicate since one must deal with singularities of varieties linked to~$W$. Proposition \ref{Rlink}, which is the main result of \S\ref{realpar}, is specific to the case where $k$ is a real closed field.

\begin{prop}[Hironaka]
\label{Hironaka}
Assume that $n>2d$, that $d\leq 3$, and that $W$ is smooth. Then, for $j\geq 4$ and $r$-tuples of integers $\ul_j\gg \dots \gg \ul_{1}\gg0$, there exists a chain of linked smooth subvarieties $W=W_0\sim_{\ul_1} W_1\sim_{\ul_2}\dots\sim_{\ul_j}W_j$ of $V$ such that $\pi|_{W_j}:W_j\to X$ is a closed embedding.
 \end{prop}

\begin{proof}
Choose general links $W=W_0\sim_{\ul_1} W_1\sim_{\ul_2}\dots\sim_{\ul_j}W_j$.
Since $d\leq 3$ and $W$ is smooth, the $W_i$ are smooth by \cite[Corollary 3.9.1]{Hironakasmoothing}.

That $\pi|_{W_j}:W_j\to X$ is a closed embedding may be checked over an algebraic closure of $k$, where it follows from the proof of \cite[Lemma 5.1.1]{Hironakasmoothing}. More precisely, define $A(W_i)\subset W_i$ to be the closed subset of those $x\in W_i$ such that ${\pi|_{W_i}:W_i\to X}$ is not a closed embedding above a neighbourhood of $\pi(x)$.  One has ${\dim(A(W_0))\leq\dim(W)\leq 3}$. Moreover, $\dim(A(W_{i+1}))\leq\dim(A(W_i))$, and the inequality is strict if $A(W_i)\neq\varnothing$, by the proof of \cite[Lemma 5.1.1]{Hironakasmoothing}. It follows that $A(W_i)=\varnothing$ for $i\geq 4$, hence that $\pi|_{W_j}:W_j\to X$ is a closed embedding.
\end{proof}

\begin{prop}
\label{propBertini}
Assume that $n>2d$, and let $\ul=(l_1,\dots,l_r)$ be an $r$-tuple of integers such that the linear systems $H^0(V,\cI_W(l_i))$ embed $V\setminus W$ in projective spaces. 
Then, for $\uF\in H^0(V,\cI_W(\ul))$ general, the link $W\sim_{\ul}W'$ associated with $\uF$ satisfies:
\begin{enumerate}[(i)]
\item The variety $S:=W'\setminus(W\cap W')$ is smooth.
\item The morphism $\pi|_{S}:S\to X$ is an embedding.
\item The subsets $\pi(S)$ and $\pi(W)$ of $X$ are disjoint.
\end{enumerate}
\end{prop}

\begin{proof}
Apply Lemma \ref{Bertini} below with $Y=V\setminus W$, $f=\pi|_{V\setminus W}$, and $F=\pi(W)$.
\end{proof}

\begin{lem}
\label{Bertini}
Let $f:Y\to X$ be a smooth morphism of smooth varieties over $k$, let $F\subset X$ be a closed subset,
and let $V_1,\dots,V_r$ be linear systems on $Y$ inducing embeddings of $Y$ in projective spaces. If one has $\dim(X)>2(\dim(Y)-r)$ and $\dim(X)>\dim(F)+\dim(Y)-r$, then, for general $\sigma_i\in V_i$, the variety ${S:=\{\sigma_i=0\}}$ is smooth, the morphism $f|_S:S\to X$ is an embedding, and ${f(S)\cap F=\varnothing}$.
\end{lem}

\begin{proof}
The smoothness of $S$ follows from the Bertini theorem. For general $\sigma_i$, the complete intersection $S\cap f^{-1}(F)$ in $f^{-1}(F)$ is empty, hence ${f(S)\cap F=\varnothing}$.

 Let $H$ be the Hilbert scheme parametrizing zero-dimensional subschemes $Z\subset Y$ of length $2$ in the fibers of $f$. Let $B\subset H\times V_1\times\dots\times V_r$ be the closed subset parametrizing tuples $([Z],\sigma_1,\dots,\sigma_r)$ such that the $\sigma_i$ vanish on $Z$. Computing
$$\dim(B)=\dim(H)+\sum_i\dim(V_i)-2r=2\dim(Y)-\dim(X)+\sum_i\dim(V_i)-2r,$$
shows that $\dim(B)<\sum_i\dim(V_i)$.
To ensure that $f|_S$ is an embedding, it suffices to choose $(\sigma_1,\dots,\sigma_r)$ outside of the image of the projection $B\to V_1\times\dots\times V_r$.
\end{proof}

\begin{prop}
\label{propinj}
Assume that $n>2d$ and that $W$ is a local complete intersection. For $j\gg 0$ and for $r$-tuples of integers $\ul_j\gg \dots \gg \ul_{1}\gg0$, there exists a chain of linked  subvarieties $W=W_0\sim_{\ul_1} W_1\sim_{\ul_2}\dots\sim_{\ul_j}W_j$ of $V$ with the property that $\pi|_{W_j}:W_j\to X$ is geometrically injective.
\end{prop}

\begin{proof}
Let $W=W_0\sim_{\ul_1} W_1\sim_{\ul_2}\dots\sim_{\ul_j}W_j$ be a general chain of links.
Define $C(W_i)\subset W_i$ to be the constructible subset of those $x\in W_i$ such that $(\pi|_{W_i})^{-1}(\pi(x))$ has more than one geometric point. Of course, one has $\dim(C(W_0))\leq\dim(W)=d$. Proposition \ref{propBertini} implies that $C(W_{i+1})\subset C(W_i)$ for all ${i\geq 0}$. We also claim that $\dim(C(W_{2i+1}))<\dim(C(W_i))$ if $2i+1\leq j$ and ${C(W_i)\neq\varnothing}$. These facts imply that $C(W_{j})=\varnothing$ if $j\geq 2^{d+1}-1$, which concludes.

It remains to prove the claim. Assume that $2i+1\leq j$ and that ${C(W_i)\neq\varnothing}$. Choose finitely many points $x_1,\dots,x_N\in C(W_i)$, including at least one in each irreducible component of the Zariski closure of $C(W_i)$. By Lemma \ref{disappear}, there exists a chain of links $W_i\sim_{\ul_i}W_{i+1}^{(s)}\sim_{\ul_{i+1}}\!\dots\sim_{\ul_{2i+1}}W^{(s)}_{2i+1}$ such that ${x_s\notin W^{(s)}_{2i+1}}$, for all ${1\leq s\leq N}$. Since $W_i\sim_{\ul_i}W_{i+1}\sim_{\ul_{i+1}}\!\dots\sim_{\ul_{2i+1}}W_{2i+1}$ corresponds to a general point of the $(i+1)$-th moduli of links of $W_i$ in the sense of \S\ref{generic}, we deduce that $x_s\notin W_{2i+1}$, hence that $x_s\notin C(W_{2i+1})$ for $1\leq s\leq N$. The chain of inclusions $C(W_{2i+1})\subset C(W_{2i})\subset\dots\subset C(W_{i})$ implies that $\dim(C(W_{2i+1}))<\dim(C(W_i))$, which proves the claim.
\end{proof}

\begin{rem}
The proof of Proposition \ref{propinj} requires to use a huge number of links (exponential in $d$). We do not know if this is necessary, or a quirk of the proof.
\end{rem}

  \subsection{Real loci} 
  \label{realpar}
  
In \S\ref{realpar}, we keep the notation of \S\ref{morphism} and assume moreover that $k=R$ is a real closed field, for instance the field $\R$ of real numbers.  

\begin{lem}
\label{doublerien}
Fix $r$\nobreakdash-tuples of integers $\ul_1=(l_{1,1},\dots, l_{1,r})$ and $\ul_2=(l_{2,1},\dots, l_{2,r})$ with $l_{2,i}-l_{1,i}$ nonnegative and even for $1\leq i\leq r$. Let $W\sim_{\ul_1}W_1$ be a link. Then there exists a link $W_1\sim_{\ul_2}W_2$ 
such that $W_2=W$ in a neighbourhood of~$V(R)$.
\end{lem}
  
\begin{proof}
Let $(u_1,\dots,u_N)$ be a basis of $H^0(V,\cO_V(1))$. The section $u:=\sum_{m=1}^N u_m^2$ does not vanish on $V(R)$. Let $v_1,\dots,v_r\in H^0(V,\cO_V(2))$ be general small deformations of $u$. They are general elements of $H^0(V,\cO_V(2))$ that do not vanish on~$V(R)$.

Let $\uF=(F_1,\dots,F_r)\in H^0(V,\cI_W(\ul_1))$ be a regular sequence defining ${W\sim_{\ul_1}W_1}$. There exist integers $a_i\geq 0$ such that $\uG:=(v_1^{a_1}F_1,\dots,v_r^{a_r}F_r)\in H^0(V,\cI_{W_1}(\ul_2))$.
Since the $v_i$ are general and since $\uF$ is a regular sequence, we see that $\uG$ is also regular sequence. Let $W_1\sim_{\ul_2}W_2$ be the link it defines.

The ideal sheaves $\langle\uF\rangle$ and $\langle\uG\rangle$ of $\cO_V$ coincide in a neighbourhood of $V(R)$ since the $v_i$ do not vanish on $V(R)$. It thus follows from the symmetry of linkage (see~\S\ref{link}) that $W_2=W$ in a neighbourhood of $V(R)$.  
\end{proof}

\begin{lem}
\label{Rlinklem}
Set $W_0:= W$.
Suppose that $W_0$ is smooth along $W_0(R)$ and that $n>2d$. Fix $r$\nobreakdash-tuples of even integers $\ul_2\gg\ul_1\gg 0$. If $R=\R$, fix a neighbourhood $\cU\subset \ci(W_0(\R),V(\R))$ of the inclusion. Then there exist links $W_0\sim_{\ul_1}W_1\sim_{\ul_2}W_2$ with the following properties.
\begin{enumerate}[(i)]
\item The variety $W_2$ is smooth along $W_2(R)$.
\item Let $D(W_i)\subset W_i$ be the subset of those $x\in W_i$ such that $\pi|_{W_i}$ is not immersive at $x$. Set $d_i:=\sup_{x\in D(W_i)(R)} \dim_x D(W_i)$. If $D(W_0)(R)\neq\varnothing$, then $d_2<d_0$.
\item If $R=\R$, there exists a diffeomorphism $\phi: W_0(\R)\isoto W_2(\R)$
such that, letting ${\iota: W_2(\R)\to V(\R)}$ denote the inclusion, one has $\iota\circ\phi\in \cU$.
\end{enumerate}
\end{lem}

\begin{proof}
Choose finitely many points $x_1,\dots,x_N\in W_0(R)$, including at least one in each irreducible component of $D(W_0)$ that has real points. By Lemma \ref{disappear}, a link $W_0\sim_{\ul_1} W_1$ corresponding to a general point of the first moduli of links of $W_0$ (in the sense of \S\ref{generic}) has the property that $x_s\notin W_1$ for $1\leq s \leq N$. Lemma \ref{doublerien}
 shows the existence of a link $W_1\sim_{\ul_2}\wW$ such that $\wW=W_0$ in a neighbourhood of~$V(R)$.
We deduce that $\wW$ is smooth along $\wW(R)=W_0(R)$.

Let $\wf:\wkW\to L_1(W_1)$ be the first moduli of links of $W_1$ with respect to the degree~$\ul_2$ (as in \S\ref{generic}) and let $b\in L_1(W_1)(R)$ be the point associated with ${W_1\sim_{\ul_2}\wW}$. Proposition~\ref{linkrel} shows that $\wkW_{b}=\wW$ and that $\wf$ is flat, hence smooth in a neighbourhood of $\wkW_b(R)$. As $\wf$ is proper, the map $\wf(R)$ is closed by \cite[Theorem~9.6]{DK2}. We deduce that there exists a Euclidean 
neighourhood $\Omega$ of $b$ in $L_1(W_1)(R)$ such that the morphism $\wf$ is smooth along $\wf(R)^{-1}(\Omega)$. 
Choose such an $\Omega$ small enough.
Since $L_1(W_1)$ is smooth and irreducible (see \S\ref{generic}), the subset $\Omega\subset L_1(W_1)$ is Zariski-dense (apply \cite[Proposition~2.8.14]{BCR}). Consequently, one may choose $a\in\Omega$ general. 
 Let $W_1\sim_{\ul_2}W_2:=\wkW_a$ be the associated link. Assertion (i) holds by our choice of $\Omega$.

 Let $D\subset \wkW$ be the closed subset of those $x\in\wkW$ such that the morphism $\pi|_{\tilde{f}^{-1}(\tilde{f}(x))}:\wkW_{\tilde{f}(x)}\to X$ is not immersive at $x$. Set $E:=D\cap (W_1\times  L_1(W_1))\subset \wkW$. By Proposition \ref{propBertini}~(ii), there is a proper closed subset $F\subset L_1(W_1)$ such that $D\subset \wf^{-1}(F)\cup (W_1\times  L_1(W_1))$ as subsets of $V\times L_1(W_1)$. Since $a$ has been chosen general, it lies outside of~$F$. We deduce the inclusion $D(W_2)\subset E_a$ of subsets of $\wkW_a$. The function $x\mapsto \dim_x(E_{\tilde{f}(x)})$ is upper semicontinuous for the Zariski topology on~$E$ by \cite[Th\'eor\`eme 13.1.3]{EGA43}, hence upper semicontinuous for the Euclidean topology on $E(R)$. Since $\wf|_E:E\to L_1(W_1)$ is proper, the map $\wf|_E(R)$ is closed by \cite[Theorem~9.6]{DK2}. If $\Omega$ has been chosen small enough, we deduce at once the inequality:
\begin{equation}
\label{ineq}
\sup_{x\in E_b(R)}\dim_x E_b\geq \sup_{x\in E_a(R)}\dim_xE_a.
\end{equation}
 As $\wkW_b=W_0$ in a neighbourhood of $V(\R)$, one has $E_b\subset D(W_0)$ in a neighbourhood of $V(\R)$. Since none of the $x_s$ belong to $W_1$, we see that no irreducible components of $D(W_0)$ that has real points is included in $E_b$. If  $D(W_0)(R)\neq\varnothing$, it follows that the left-hand side of (\ref{ineq}) is $<d_0$. On the other hand, the right-hand side of (\ref{ineq}) is $\geq d_2$ since $D(W_2)\subset E_a$. The inequality $d_2<d_0$ follows, proving (ii).

 If $R=\R$, one may assume $\Omega$ to be connected. 
Ehresmann's theorem applied to the proper submersion $\wf(\R)|_{\wf(\R)^{-1}(\Omega)}:\wf(\R)^{-1}(\Omega)\to\Omega$ yields a diffeomorphism $\psi:\Omega\times \wkW_{b}(\R)\isoto \wf(\R)^{-1}(\Omega)$ compatible with the projections to~$\Omega$. If $\Omega$ has been chosen small enough, the composition $$W_0(\R)=\wkW_{b}(\R)\xrightarrow{\psi_a}\wkW_{a}(\R)=W_2(\R)\xrightarrow{\iota}V(\R)$$ belongs to~$\cU$ for all $a\in\Omega$. Assertion~(iii) is proven.
\end{proof}

\begin{prop}
\label{Rlink}
Suppose that $W$ is smooth along $W(R)$ and that $n>2d$. Let $j\gg 0$ be even and let ${\ul_j\gg \dots \gg\ul_1\gg 0}$ be $r$-tuples of even integers. 
Define ${f_j:\kW_j\to L_j(W)}$  as in~\S\ref{generic}. Then there exists a nonempty subset ${\Omega\subset L_j(W)(R)}$ which is open for the Euclidean topology
such that the following holds for all~$b\in\Omega$.
\begin{enumerate}[(i)]
\item The variety $\kW_{j,b}$ is smooth along $\kW_{j,b}(R)$.
\item The morphism $\pi|_{\kW_{j,b}}$ is immersive along $\kW_{j,b}(R)$.
\end{enumerate}
If moreover $R=\R$ and $\cU\subset \ci(W(\R),V(\R))$ is a neighbourhood of the inclusion, one may ensure that the following holds.
\begin{enumerate}[(i)]
\setcounter{enumi}{2}
\item There exists a diffeomorphism $\phi_b: W(\R)\isoto\kW_{j,b}(\R)$
such that, letting ${\iota_b:\kW_{j,b}(\R)\to V(\R)}$ denote the inclusion, one has $\iota_b\circ\phi_b\in \cU$.
\end{enumerate}
\end{prop}

\begin{proof}
Choose $j\geq 2d+2$ even. Applying $j/2$ times Lemma \ref{Rlinklem} shows the existence of a chain of linked subvarieties $W=W_0\sim_{\ul_1} W_1\sim_{\ul_2}\dots\sim_{\ul_j}W_j$ of~$V$ such that $W_j$ is smooth along $W_j(R)$ and such that $\pi|_{W_j}$ is immersive along $W_j(R)$. Moreover, if $R=\R$, Lemma \ref{Rlinklem} ensures the existence of a diffeomorphism $\phi:W(\R)\isoto W_j(\R)$  such that, letting $\iota:W_j(\R)\to V(\R)$ denote the inclusion, one has $\iota\circ\phi\in\cU$.

Let $a\in L_j(W)(R)$ be the point associated with $W=W_0\sim_{\ul_1} W_1\sim_{\ul_2}\dots\sim_{\ul_j}W_j$.
 Proposition~\ref{linkrel} shows that $\kW_{j,a}=W_j$ and that $f_j$ is flat, hence smooth in a neighbourhood of $\kW_{j,a}(R)$. As $f_j$ is proper, the map $f_j(R)$ is closed by \cite[Theorem~9.6]{DK2}. We deduce the existence of a
neighourhood $\Omega$ of $a$ in $L_j(W)(R)$ such that $f_j$ is smooth along $f_j(R)^{-1}(\Omega)$. Assertion (i) follows. So does assertion~(ii) after maybe shrinking $\Omega$. If $R=\R$, that assertion (iii) holds after shrinking $\Omega$ further follows from Ehresmann's theorem applied to the proper submersion $f_j(\R)|_{f_j(\R)^{-1}(\Omega)}:f_j(\R)^{-1}(\Omega)\to\Omega$.
\end{proof}

 \section{Smoothing by linkage}
 \label{sectionapprox}
 
 Let us apply linkage theory as developed in \S\ref{linkage} to prove Theorems \ref{Chow1}, \ref{Chow2} and~\ref{thC1}.

  \subsection{Main statement}
  
  Here is the technical result from which the main theorems of \S\ref{sectionapprox} will follows.
  
   \begin{prop}
   \label{mainapprox}
 Let $g:W\to X$ be a morphism of smooth projective varieties over a real closed field $R$. Let $d$ and $n$ be the dimensions $W$ and $X$ and assume that $n>2d$. If $R=\R$, fix a neighbourhood $\cV\subset\ci(W(\R),X(\R))$ of $g(\R)$. Then there exists a closed subvariety $i:Y\hookrightarrow X$ of dimension $d$ with the following properties.
 \begin{enumerate}[(i)]
 \item The variety $Y$ is smooth along $Y(R)$.
 \item If $d\leq 3$, then $Y$ is smooth.
 \item The class $g_*[W]-[Y]\in\CH_d(X)$ is a linear combination of classes of smooth closed subvarieties of $X$ with empty real loci.
  \item If $R=\R$, there is a diffeomorphism $\psi: W(\R)\isoto Y(\R)$ such that $i(\R)\circ\psi\in\cV$.
 \end{enumerate}
 \end{prop}
 
 \begin{proof}
 Define $V:=X\times W$, let $\pi:V\to X$ be the first projection and let $\cO_V(1)$ be a very ample line bundle on $V$. Let $r$ be the codimension of the closed embedding $(g,\Id):W\hookrightarrow V$. If $R=\R$, define $\cU:=\{\xi\in\ci(W(\R),V(\R))\,|\,\pi(\R)\circ\xi\in\cV\}$, which is a neighbourhood of the inclusion $W(\R)\hookrightarrow V(\R)$ by \cite[Theorem~11.4]{Michor}.
 
 Choose  an even integer $j\gg 0$ and $r$-tuples of even integers ${\ul_j\gg \dots \gg \ul_{1}\gg0}$. Let ${f_j:\kW_j\to L_j(W)}$ be the $j$-th moduli of links of $W$ as in~\S\ref{generic}, and choose $\Omega\subset L_j(W)(R)$ be as in Proposition \ref{Rlink}.
  Since $L_j(W)$ is smooth and irreducible (see \S\ref{generic}), the subset $\Omega\subset L_j(W)$ is Zariski-dense (apply \cite[Proposition~2.8.14]{BCR}). Consequently, one may choose a general point $b\in\Omega$. Define ${Y:=\pi(\kW_{j,b})}$ with inclusion $i:Y\hookrightarrow X$. By Proposition \ref{propinj}, the morphism $\pi|_{\kW_{j,b}}:\kW_{j,b}\to Y$ is geometrically injective. By Proposition \ref{Rlink} (i)-(ii), in a neighbourhood of $\kW_{j,b}(R)$, the variety $\kW_{j,b}$ is smooth and the morphism $\pi|_{\kW_{j,b}}$ is immersive. These facts show that $\theta:=\pi(R)|_{\kW_{j,b}(R)}:\kW_{j,b}(R)\to Y(R)$ is bijective and that $Y$ is smooth along the image of $\theta$. This proves (i). The argument also shows that $\theta$ is a diffeomorphism.  
Consequently, if $R=\R$, one may take $\psi:=\theta\circ\phi_b$, where $\phi_b$ is as in Proposition~\ref{Rlink}~(iii), which proves (iv). If $d\leq 3$, Proposition~\ref{Hironaka} shows that $\kW_{j,b}$ is smooth and that $\pi|_{\kW_{j,b}}$ is a closed immersion, proving~(ii). 

It remains to prove (iii). The chain of links relating $W$ and $\kW_{j,b}$ shows that the difference $[W]-[\kW_{j,b}]\in\CH_d(V)$ is a multiple of $(2\lambda)^r\in\CH^r(V)=\CH_d(V)$, where $\lambda:=c_1(\cO_V(1))\in\CH^1(V)$. Consequently, $\pi_*[W]-\pi_*[\kW_{j,b}]=g_*[W]-[Y]$ is a multiple of $\pi_*((2\lambda)^r)\in\CH_d(X)$. Let $(u_1,\dots,u_N)$ be a basis of $H^0(V,\cO_V(1))$. Proposition \ref{propBertini} applied with $W=\varnothing$, with $l_i=2$, and with the $F_i$ chosen to be general small deformations of $\sum_{m=1}^Nu_m^2$, shows that $\pi_*((2\lambda)^r)$ is the class of a smooth closed subvariety of $X$ with empty real locus, which concludes.
\end{proof}
 
\begin{rem}
Over a general real closed field, one could replace Proposition \ref{mainapprox} (iv) by the assertion that there exists a Nash diffeomorphism $\psi:W(R)\isoto Y(R)$ such that $g(R):W(R)\to X(R)$ and $i(R)\circ \psi:Y(R)\to X(R)$ are Nash homotopic.
The proof is identical, replacing the use of Ehresmann's theorem in the proofs of Lemma~\ref{Rlinklem}~(iii) and Proposition \ref{Rlink} by its Nash analogue proven by Coste and Shiota \cite[Theorem 2.4 (iii)']{CS}.
\end{rem}

 \subsection{Low-dimensional cycles}
 
 We first give applications to the Chow groups of smooth projective varieties over real closed fields.
 
 \begin{thm}
 \label{Chowth}
 Let $X$ be a smooth projective variety of dimension $n$ over a real closed field $R$. For $n>2d$, the Chow group $\CH_d(X)$ is generated by classes of closed integral subvarieties of~$X$ which are smooth along their real loci.
 \end{thm}
 
 \begin{proof}
 Let $Z\subset X$ be a closed integral subvariety of $X$, let $W\to Z$ be a resolution of singularities, and let $g:W\to X$ be the induced morphism. Proposition~\ref{mainapprox} furnishes a closed subvariety $Y\subset X$ which is smooth along $Y(R)$ and such that $g_*[W]-[Y]\in\CH_d(X)$ is a linear combination of classes of smooth closed subvarieties of $X$. This proves the theorem.
 \end{proof}
 
 \begin{thm}
 \label{Chowth2}
 Let $X$ be a smooth projective variety of dimension $n$ over $\R$. For $n>2d$, the group $\Ker \big[\cl_{\R}:\CH_d(X)\to H_d(X(\R),\Z/2)\big]$ is generated by classes of closed integral subvarieties of~$X$ with empty real loci.
 \end{thm}
 
 \begin{proof}
 Ischebeck and Sch\"ulting \cite[Main Theorem 4.3]{IS} have shown that the group $\Ker \big[\cl_{\R}:\CH_d(X)\to H_d(X(\R),\Z/2)\big]$ is generated by classes of closed integral subvarieties of~$Z\subset X$ with the property that $Z(\R)$ is not Zariski-dense in $Z$.
 Let $W\to Z$ be a resolution of singularities of such a subvariety, and let $g:W\to X$ be the induced morphism. 
Since the real locus of a smooth irreducible variety over $\R$ is empty or Zariski-dense,
one has $W(\R)=\varnothing$. 

By Proposition~\ref{mainapprox}, there is a closed subvariety $Y\subset X$ with $Y(\R)\simeq W(\R)=\varnothing$ and such that $g_*[W]-[Y]\in\CH_d(X)$ is a linear combination of classes of closed subvarieties of $X$ with empty real loci. This concludes the proof.
\end{proof}

\begin{rem}
\label{remrealclosed}
The proof of Theorem \ref{Chowth2} does not extend to general real closed fields because of  its use of \cite[Main Theorem~4.3]{IS}. As a matter of fact, the statement of Theorem \ref{Chowth2} does not hold over the real closed field $R:=\cup_n\R((t^{1/n}))$, for any $d>0$, as we show in the next proposition.
\end{rem}

 \begin{prop}
 \label{realclosed}
 For all $c,d\geq 1$, there exist a smooth projective variety  $X_{c,d}$ of dimension $c+d$ over~$R:=\cup_n\R((t^{1/n}))$ and a class $\beta_{c,d}\in\CH_d(X_{c,d})$ such that:
 \begin{enumerate}[(i)]
 \item One has $\cl_R(\beta_{c,d})=0\in H_d(X_{c,d}(R),\Z/2)$.
 \item For all identities $\beta_{c,d}=\sum_{i\in I} n_i[Z_i]\in\CH_d(X_{c,d})$ with $n_i\in\Z$ and $Z_i\subset X$ integral, there exists $i\in I$ such that
 $n_i$ is odd and $Z_i(\R)$ is Zariski-dense in~$Z_i$.
 \end{enumerate}
\end{prop}
 
 \begin{proof}
 Such $X_{1,1}$ and $\beta_{1,1}$ have been constructed in \cite[Propositions 9.17 and~9.19~(ii)]{BW2} (they form a counterexample to Br\"ocker's EPT 
theorem over the field $R$).
Let $x\in \P^{c-1}(R)$ and $y\in \P^{d-1}(R)$ be general points.
 One may then define $X_{c,d}:=X_{1,1}\times \P^{c-1}_R\times \P^{d-1}_R$ and $\beta_{c,d}:=pr_1^*\beta_{1,1}\cdot pr_2^*[x]$. The required property of $(X_{c,d},\beta_{c,d})$ follows from that of $(X_{1,1},\beta_{1,1})$, and from the equation $\beta_{1,1}=(pr_1)_*(\beta_{c,d}\cdot pr_3^*[y])$.
 \end{proof}

 \subsection{Approximation of submanifolds}
 \label{parapp}
 
 We now give an application to the existence of algebraic approximations for submanifolds of the real locus of a smooth projective variety~$X$ over $\R$. 
 We refer to \S\ref{approximation} for the definition of the algebraic cobordism group $MO^{\alg}_d(X(\R))$.
  
  \begin{thm}
  \label{approxth}
  Let $X$ be a smooth projective variety of dimension $n$ over $\R$, and let $j:M\hookrightarrow X(\R)$ be a closed $\ci$ submanifold of dimension $d$. If $n>2d$, the following properties are equivalent.
  \begin{enumerate}[(i)]
  \item\label{i}
One has $[j]\in MO^{\alg}_d(X(\R))$.
  \item  \label{ii}
  For all neighbourhoods $\cU\subset\ci(M,X(\R))$ of $j$, there exist a closed $d$-dimen\-sional subvariety $i:Y\hookrightarrow X$ smooth along $Y(\R)$ and a diffeomorphism ${\phi:M\isoto Y(\R)}$ such that $i(\R)\circ\phi\in\cU$.
  \end{enumerate}
  If moreover $d\leq 3$, one may choose $Y$ to be smooth in assertion (ii).
   \end{thm}
 
 \begin{proof}
 Assume that (i) holds and let $\cU$ be as in (ii). A relative variant of the Nash--Tognoli theorem  (see \cite[Proposition~4.1]{BTvb} or \cite[Proposition 0.2]{AKIHES}) shows the existence of a morphism $g:W\to X$ of smooth projective varieties over~$\R$ and of a diffeomorphism $\chi:M\isoto W(\R)$ such that $g(\R)\circ\chi\in\cU$. 
 Assertion~(ii) and the last statement of Theorem \ref{approxth} now follow by applying Proposition~\ref{mainapprox} to $\cV:=\{\xi\circ\chi^{-1},\,\xi\in\cU\}\subset\ci(W(\R),X(\R))$ and by defining $\phi:=\psi\circ\chi$.
 
 Suppose conversely that (ii) holds. Applying it to a small enough neighbourhood~$\cU\subset\ci(M,X(\R))$ of $j$ shows that $j$ is homotopic, hence cobordant, to a $\ci$ map of the form $i(\R)$, by \cite[Proposition 4.4.4]{Wall}. To get (i), take $W\to Y$ to be a resolution of singularities which is an isomorphism over $Y(\R)$, define $g:W\to X$ be the induced morphism, and note that $j$ is cobordant to $g(\R)$.
 \end{proof}

 \section{Complex cobordism and Chern numbers}
 \label{complexcobordism}
 
After a short review of cobordism theory (in \S\ref{cob}) and of its relation with characteristic classes (in \S\ref{swc}), we study the top Segre class in \S\ref{partop}, our goal being Theorem \ref{divtop}.

 \subsection{The cobordism rings}
 \label{cob}
 
 Two compact $\ci$ manifolds $M_1$ and~$M_2$ of dimension~$n$ are said to be \textit{cobordant} if there exists a compact $\ci$ manifold with boundary~$C$ and a diffeomorphism $\partial C\simeq M_1\cup M_2$. Let $MO_n$ be the set of cobordism classes of such manifolds, and define $MO_*:=\bigoplus_{n\geq 0}MO_n$.
 
Let $M$ be a $\ci$ manifold. A \textit{stably almost complex structure} on $M$ is a complex structure $J$ on the real vector bundle $T_M\oplus \R^k$ for some $k\geq 0$, modulo the equivalence relation generated by $(T_M\oplus \R^k,J)\simeq(T_M\oplus \R^{k+2}=T_M\oplus \R^k\oplus\C,(J,i))$. Two $n$-dimensional stably almost complex compact $\ci$ manifolds $M_1$ and~$M_2$ are said to be \textit{complex cobordant} if there exists a stably almost complex compact $\ci$ manifold with boundary $C$ and a diffeomorphism $\partial C\simeq M_1\cup M_2$ compatible with the stably almost complex structures.  Let $MU_n$ be the set of complex cobordism classes of such manifolds, and define $MU_*:=\bigoplus_{n\geq 0}MU_n$.

Disjoint union and cartesian product endow $MO_*$ and $MU_*$ with graded ring structures: they are the \textit{unoriented cobordism ring} and the \textit{complex cobordism ring}. Thom \cite[Th\'eor\`eme IV.12]{Thom} and Milnor \cite{MilnorMU} (see also Quillen \cite[Theorem~6.5]{Quillen}) have computed that $MO_*\simeq\Z/2[x_d]_{d\neq 2^k-1}$ and $MU_*\simeq\Z[t_{2d}]_{d\geq 1}$, where $x_d\in MO_d$ and $t_{2d}\in MU_{2d}$. A comprehensive treatment of these results may be found in \cite{Kochman}.  As was noted by Milnor \cite[Lemma 1]{MilnorSW}, a striking consequence of Thom's computation is that any element of $MO_*$ may be represented by the real locus of a smooth projective variety over $\R$ (a disjoint union of products of projective spaces and Milnor hypersurfaces).

 Let $\phi:MU_*\to MO_*$ be the graded ring homomorphism forgetting the stably almost complex structures. Milnor \cite[Theorem 1]{MilnorSW} has shown that the image of~$\phi$ consists exactly of the squares in $MO_*$, hence that there exists a surjective ring homomorphism $\psi:MU_*\to MO_*$ such that $\phi(x)=\psi(x)^2$ for all $x\in MU_*$.  
 
 \begin{lem}
 \label{lemker}
The ideal $\ker(\psi)\subset MU_*$ is generated by $2$ and by $(\ker(\psi)_{2^{k+1}-2})_{k\geq 1}$. 
 \end{lem}
 
 \begin{proof}
We use the isomorphisms $MO_*\simeq\Z/2[x_d]_{d\neq 2^k-1}$ and $MU_*\simeq\Z[t_{2d}]_{d\geq 1}$. Since $\psi:MU_*\to MO_*$ is surjective, $\psi(t_{2d})-x_d\in MO_d$ is decomposable for $d\neq 2^k-1$, and $\psi(t_{2d})$ is of course decomposable for $d=2^k-1$. We deduce from the surjectivity of $\psi$ the existence of $t'_{2d}\in MU_{2d}$ such that $t_{2d}-t'_{2d}$ is decomposable (so that $MU_*=\Z[t'_{2d}]_{d\geq 1}$) and such that $\psi(t'_{2d})=x_d$ if $d\neq 2^k-1$ and $\psi(t'_{2d})=0$ otherwise. It is now clear that $\ker(\psi)$ is generated by $2$ and by the $(t'_{2^{k+1}-2})_{k\geq 1}$.
 \end{proof}

 \subsection{Stiefel--Whitney and Chern numbers}
 \label{swc}
 
 Let $M$ be a compact~$\ci$ manifold of dimension $n$, and let $w_r(M)\in H^r(M,\Z/2)$ be the $r$-th Stiefel--Whitney class of its tangent bundle.
For a sequence of nonnegative integers $I=(i_1,i_2,\dots)$ with $|I|:=\sum_r ri_r=n$, we define $w_I(M):=\langle \prod_r w_r(M)^{i_r},[M]\rangle\in\Z/2$, where $[M]$ is the fundamental class of $M$. The $w_I(M)$ are the \textit{Stiefel--Whitney numbers} of $M$. Thom has shown that they only depend on the cobordism class of~$M$, and that they determine this cobordism class \cite[Th\'eor\`emes~IV.3 and~IV.11]{Thom}.

 Similarly, if $M$ is a stably almost complex compact $\ci$ manifold of dimension~$n$, we let $c_r(M)\in H^{2r}(M,\Z)$ denote the $r$-th Chern class of its stable tangent bundle, and we define the \textit{Chern numbers} $c_I(M):=\langle \prod_r c_r(M)^{i_r},[M]\rangle\in \Z$ of $M$ for $|I|=n/2$. These numbers only depend on the complex cobordism class of $M$, and determine it (see \cite[Theorem 6.5]{Quillen}).
 
 \begin{lem}
 \label{cw}
 For all $y\in MU_{2d}$ and all $I=(i_1,i_2,\dots)$ with $|I|=d$, the reduction modulo $2$ of $c_I(y)$ is equal to $w_I(\psi(y))$.
 \end{lem}
 
 \begin{proof}
 Represent $y$ by a stably almost complex compact $\ci$ manifold $M$ and $\psi(y)$ by a compact $\ci$ manifold $N$. For $I'=(0,i_1,0,i_2,\dots)$, one has 
 $$c_I(M)=w_{I'}(M)=w_{I'}(N\times N)=w_I(N)\in\Z/2,$$
 where the first equality holds by \cite[Problem 14-B]{MS}, the second since $M$ is cobordant to $N\times N$, and the third is \cite[Lemma~2]{MilnorSW}.
\end{proof}

The relation $(\sum_r c_r(M))(\sum_r s_r(M))=1$ defines the \textit{Segre classes} or \textit{normal Chern classes} $s_r(M)\in H^{2r}(M,\Z)$ of $M$.
Pairing the top Segre class with $[M]$ yields a morphism ${s_d:MU_{2d}\to\Z}$ which is a linear combination of Chern numbers. This characteristic number is multiplicative in the following sense.
 
 \begin{lem}
 \label{mult}
 For $y\in MU_{2d}$ and $y'\in MU_{2d'}$, one has $s_{d+d'}(yy')=s_{d}(y)s_{d'}(y')$.
 \end{lem}
 
 \begin{proof}
 Represent $y$ and $y'$ by stably almost complex compact $\ci$ manifolds and apply the Whitney sum formula \cite[(14.7)]{MS}.
 \end{proof}

 \subsection{Divisibility of the top Segre class}
 \label{partop}

 In \cite{RT},  Rees and Thomas study divisibility properties of some Chern numbers. In Theorem \ref{RT}, we recall one of their results, which we complement in Theorem \ref{divtop}. Recall from \S\ref{notation} that we let $\alpha(m)$ be the number of ones in the dyadic expansion of $m$.
 
 \begin{thm}[Rees--Thomas]
 \label{RT}
For $d\geq 0$ and $e\geq 1$, the function ${s_d:MU_{2d}\to\Z}$ is divisible by $2^e$ if and only if $\alpha(d+e-1)>2(e-1)$.
 \end{thm}
 
 \begin{proof}
  Rees and Thomas \cite[Theorem 3]{RT} show that $s_d:MU_{2d}\to\Z$ is divisible by $2^e$ if and only if $\alpha(d+f)>2f$ for all $0\leq f\leq e-1$, and it is easily verified that $\alpha(d+f)>2f$ implies that $\alpha(d+f-1)>2(f-1)$ for all $f\geq 1$.
 \end{proof}
 
 We point out for later use an easy corollary of Lemma~\ref{mult} and Theorem~\ref{RT}.
 
 \begin{cor}
 \label{cor4}
 For $d\geq 1$, the function $s_d:MU_{2d}\to\Z$ takes even values, and takes values divisible by $4$ on decomposable elements.
 \end{cor} 
 
 Here is the main result of \S\ref{complexcobordism}.
 
 \begin{thm} 
 \label{divtop}
 Let $d\geq 0$ and $e\geq 1$ be such that $\alpha(d+e-1)>2(e-1)$. Then the function $\frac{s_d}{2^e}:MU_{2d}\to\Z$ coincides modulo $2$ with an integral linear combination of Chern numbers if and only if $\alpha(d+e)\geq 2e$.
 \end{thm}
 
 \begin{proof}
Assume first that $\alpha(d+e)<2e$, and let $d+e=2^{a_1}+\dots+2^{a_f}$ be the dyadic expansion of $d+e$, with $f\leq 2e-1$.
Define $d_1:=2^{a_1}+\dots+2^{a_{f-1}}-(e-1)$ and $d_2:=2^{a_f}-1$. One has $\alpha(d_1+e-1)=f-1\leq 2e-2$. It thus follows from Theorem~\ref{RT} that there exists $y_{1}\in MU_{2d_1}$ such that $s_{d_1}(y_1)$ is not divisible by~$2^{e}$. Theorem~\ref{RT} 
also shows the existence of $y_2\in MU_{2d_2}$ such that $s_{d_2}(y_2)$ is not divisible by $4$. 
We deduce from Lemma \ref{mult} that $s_{d}(y_1y_2)/2^e\in\Z$ is odd.

 Since the map ${\psi:MU_*\to MO_*}$ is surjective and $MO_{2^{a_f}-1}$ contains no indecomposable element (see \S\ref{cob}), there exists a decomposable element $z_2\in MU_{2d_2}$ with $\psi(z_2)=\psi(y_2)$. By Corollary \ref{cor4}, one may replace $y_2$ with $y_2-z_2$ and thus assume that $\psi(y_2)=0$. But then $\psi(y_1y_2)=\psi(y_1)\psi(y_2)=0$, which shows, in view of Lemma \ref{cw}, that all the Chern numbers of $y_1y_2$ are even. Consequently, $s_{d}(y_1y_2)/2^e \pmod 2$ cannot be a linear combination with $\Z/2$ coefficients of Chern numbers of $y_1y_2$. We have thus proven the direct implication of the theorem.

Assume now that $\alpha(d+e)\geq 2e$. Let $k$ be such that $1\leq 2^k-1\leq d$.
 We claim that $MU_{2^{k+1}-2}\cdot MU_{2d-2^{k+1}+2}$ is included in the kernel of the morphism ${\chi:MU_{2d}\to\Z/2}$ obtained by reducing $\frac{s_d}{2^e}:MU_{2d}\to\Z$ modulo $2$. To see it, choose $y\in MU_{2^{k+1}-2}$ and $z\in MU_{2d-2^{k+1}+2}$. We now compute $$\alpha(d-2^k+1+(e-1))=\alpha(d-2^k+e)\geq \alpha(d+e)-1>2(e-1),$$
 and Theorem \ref{RT} shows that $s_{d-2^k+1}(z)$ is divisible by $2^e$. Since $s_{2^k-1}(y)$ is even by Corollary \ref{cor4}, Lemma~\ref{mult} shows that $s_d(yz)$ is divisible by $2^{e+1}$, as wanted.

We deduce from Lemma \ref{lemker} that the kernel of ${\psi:MU_{2d}\to MO_d}$ is included in the kernel of $\chi$, hence that $\chi=\mu\circ\psi$ for some morphism~$\mu:MO_d\to\Z/2$. Since a class in $MO_d$ is determined by its Stiefel--Whitney numbers (see \S\ref{swc}), the morphism~$\mu$ is a linear combination of Stiefel--Whitney numbers. Lemma \ref{cw} now implies that~$\chi$ is the reduction modulo $2$ of an integral linear combination of Chern numbers, which concludes the proof.
\end{proof}

\begin{ex}
\label{Noether}
 The first interesting case of Theorem \ref{divtop} is $d=2$ and $e=1$. Since $\alpha(d+e)=\alpha(3)=2=2e$, it predicts the existence of an integral linear combination of Chern numbers which coincides modulo $2$ with $\frac{s_2}{2}:MU_4\to \Z$. 
 
 We claim that this linear combination may be chosen to be $c_2$. Indeed, $MU_4$ is generated by classes of projective complex surfaces (see \cite[II, Corollary~10.8]{Adams}), and for such a surface $S$, our claim follows from Noether's formula
 $$s_2(S)=(c_1^2-c_2)(S)=12\chi(S,\cO_S)-2c_2(S).$$
\end{ex}

 \section{The double point class}
 \label{doublepoint}
  
We use the results of \S\ref{complexcobordism} in combination with a double point formula. We give applications to Chow groups in \S\ref{high}, proving Theorems \ref{Chow3}, \ref{Chow5} and \ref{Chow4}. We also construct new examples of submanifolds of real loci of smooth projective varieties over~$\R$ without algebraic approximations in \S\ref{subnoapp}, proving Theorems \ref{thC2} and \ref{thP}.

 \subsection{A consequence of Fulton's double point formula}
 \label{dbpt}

Formulas for the rational equivalence class of the double point locus of a morphism go back to Todd \cite[(7.01)]{Todd} and Laksov \cite[Theorem 26]{Laksov} under strong assumptions on the morphism. The following proposition is an application of a refined double point formula of Fulton \cite[Theorem 9.3]{Fulton}, which is valid for an arbitrary morphism.
 
 \begin{prop}
 \label{FL}
 Let $g:W\to X$ be a morphism of smooth projective varieties over $\R$. Let $d$ be the dimension of $W$ and assume that $X$ has dimension $2d$. Let $N_{W/X}:=[g^*T_X]-[T_W]$ be the virtual normal bundle of $g$.
 \begin{enumerate}[(i)]
 \item If $g$ is an embedding, then 
 \begin{equation}
 \label{modrien}
 \deg ((g_*[W])^2)=\deg(c_d(N_{W/X})).
 \end{equation}
 \item If $g$ is an embedding in a neighbourhood of $g(\C)^{-1} (X(\R))$, then 
 \begin{equation}
 \label{mod4}
 \deg ((g_*[W])^2)\equiv\deg(c_d(N_{W/X}))\pmod 4.
 \end{equation}
 \end{enumerate}
 \end{prop}
 
 \begin{proof}
 Let $D(g)\subset W$ be the closed subset consisting of those $x\in W$ such that $g$ is not an embedding above a neighbourhood of $g(x)$. Let $\D(g)\in \CH_0(D(g))$ be the double point class of $g$ defined in \cite[\S 9.3]{Fulton}. By \cite[Theorem 9.3]{Fulton}, one has 
 \begin{equation*}
 \D(g)=g^*g_*[W]-c_d(N_{W/X})\in \CH_0(W).
 \end{equation*}
 
Since $g_*g^*g_*[W]=(g_*[W])^2$ by the projection formula, we deduce that 
 \begin{equation}
 \label{eqFL}
 \deg(g_*\D(g))=\deg((g_*[W])^2)-\deg(c_d(N_{W/X}))\in\Z.
 \end{equation}
 
In case (i), one has $D(g)=\varnothing$, hence $\D(g)=0$, and (\ref{eqFL}) implies (\ref{modrien}).

Define $\oD(g):=g(D(g))$. By \cite[Example 9.3.14]{Fulton}, there exists a $0$-cycle $\overline{\D}(g)\in \CH_0(\oD(g))$ such that $g_*\D(g)=2\overline{\D}(g)\in \CH_0(\oD(g))$. The hypothesis of (ii) implies that $\oD(g)\subset X$ has no real points, hence that $\deg(g_*\D(g))=2\deg(\overline{\D}(g))$ is divisible by~$4$. The desired congruence (\ref{mod4}) now follows from~(\ref{eqFL}).
 \end{proof}
 
 \begin{rem}
 Proposition \ref{FL} (i) is of course much easier than Proposition \ref{FL} (ii). It follows for instance from \cite[Corollary 6.3]{Fulton}.
 \end{rem}

 \subsection{Weil restrictions of scalars and quotients of abelian varieties}
 
 We gather here the geometric constructions that will be used in \S\ref{high}.
 
We let $(A,\lambda)$ denote a very general principally polarized abelian variety of dimension $d\geq 1$ over $\C$.
Let $e_1,\dots,e_d,f_1,\dots,f_d\in H^1(A(\C),\Z)$ be a basis such that the principal polarization $\lambda\in H^2(A(\C),\Z)=\extp^2 H^1(A(\C),\Z)$ of $A$ is equal to $\sum_i e_i\wedge f_i$. Denote by $(A',\lambda'=\sum_ie'_i\wedge f'_i)$ another copy of $(A,\lambda)$, which we identify with the dual of $A$ by means of the principal polarization. Let $\mu\in H^2(A(\C)\times A'(\C),\Z)$ be  the class of the Poincar\'e bundle. As in \cite[Lemma 14.1.10]{BL} (whose notation is different from ours), one computes that 
$\mu=\sum_i (e_i\wedge f_i'+e'_i\wedge f_i)$.

\begin{lem}
\label{HodgesquareQ}
The subring $\Hdg^*(A(\C)\times A'(\C),\Q)\subset H^*(A(\C)\times A'(\C),\Q)$ of Hodge classes is generated by $\lambda$, $\lambda'$ and $\mu$.
\end{lem}

\begin{proof}
Let $V:=H_1(A(\C),\Q)$. In \cite{Murty}, two subgroups $\Hod(A)\subset L(A)$ of the symplectic group $\Sp(V,\lambda)$ are considered. As $A$ is very general, the subgroup $\Hod(A)$ is equal to $\Sp(V,\lambda)$ (see \cite[Proposition~17.4.2]{BL}). It follows that $\Hod(A)=L(A)=\Sp(V,\lambda)$ and that $\End(A)_{\Q}=\Q$ (see \textit{e.g.} \cite[\S 17.6~(3)]{BL}). This implies that $A$ has no factors of type III (in the sense of \cite[p.~199]{Murty}). We can now apply \cite[Theorem 3.1]{Murty} to show that $\Hdg^*(A(\C)\times A'(\C),\Q)$ is generated as a ring by $\Hdg^2(A(\C)\times A'(\C),\Q)$.

The K\"unneth formula induces an isomorphism of Hodge structures:
$$H^2(A(\C)\times A'(\C),\Q)\simeq H^2(A(\C),\Q)\oplus H^2(A'(\C),\Q)\oplus \End(H^1(A(\C),\Q)).$$

The $\Q$-vector spaces $\Hdg^2(A(\C),\Q)$ and  $\Hdg^2( A'(\C),\Q)$ are respectively generated by $\lambda$ and $\lambda'$, by Mattuck's theorem \cite[Theorem 17.4.1]{BL}. As $\End(A)_{\Q}=\Q$, the $\Q$-vector space of Hodge classes in $\End(H^1(A(\C),\Q))$ is one-dimensional, generated by $\mu$. 
This concludes the proof.
\end{proof}

In the next lemmas, we let $\Z/2$ act on $A\times A'$ by exchanging the two factors.  We define $D:=\Hdg^{2d}(A(\C)\times A'(\C),\Z)$ and $E:=D^{\Z/2}$, and we consider the subgroup $F\subset H^1(A(\C)\times A'(\C),\Z)$ generated by $e_2,\dots,e_d$, $f_1,\dots,f_d$, $e'_2,\dots,e'_d$, $f'_1,\dots,f'_d$, $2e_1, e_1+e'_1,2e'_1$. We also define $\varepsilon_{k,l,m}:=\frac{\lambda^k}{k\,!}\cdot\frac{(\lambda')^{l}}{l\,!}\cdot \frac{\mu^m}{m\,!}$, 
and we set $\zeta_{k,l,m}:=\varepsilon_{k,l,m}+\varepsilon_{l,k,m}$ if $k>l$ and $\zeta_{k,k,m}:=\varepsilon_{k,k,m}$. 

\begin{lem}
\label{independence}
\begin{enumerate}[(i)]
\item The classes $\varepsilon_{k,l,m}$,
where $(k,l,m)$ ranges over the triples of nonnegative integers such that $k+l+m=d$,
form a $\Z$-basis of $D$.
\item The classes $\zeta_{k,l,m}$, where $(k,l,m)$ ranges over the triples of nonnegative integers such that $k\geq l$ and $k+l+m=d$, form a $\Z$-basis of $E$.
\item The subgroups $E\cap \extp^{2d}F$ and $2E$ of $H^{2d}(A(\C)\times A'(\C),\Z)$ are equal.
\end{enumerate}
\end{lem}

\begin{proof}
By Lemma \ref{HodgesquareQ}, the integral classes listed in (i) span $\Hdg^{2d}(A(\C)\times A'(\C),\Q)$ as a $\Q$-vector space. To prove (i), it remains to show that these classes are linearly independent, and that they span a primitive sublattice of $\Hdg^{2d}(A(\C)\times A'(\C),\Z)$. Both assertions follow from the fact that, when one decomposes them in a basis of $H^{2d}(A(\C)\times A'(\C),\Z)=\extp^{2d}(\Z e_1\oplus\dots\oplus\Z f'_d)$ consisting of wedges of elements of $\{e_1,\dots, f'_d\}$, the coefficient of the basis element which equals
\begin{equation}
\label{basiselement}
\big(\bigwedge_{i=1}^{k}e_i\wedge f_i\big)\wedge \big(\bigwedge_{i=k+1}^{k+l}e'_i\wedge f'_i\big)\wedge\big(\bigwedge_{i=k+l+1}^{d}e_i\wedge f'_i\big)
\end{equation}
up to a sign is nonzero only in the decomposition of $\varepsilon_{k,l,m}$.

Assertion (ii) follows from (i) and from the fact that the $\Z/2$-action exchanges~$\lambda$ and $\lambda'$ and preserves $\mu$. 

We now prove (iii). 
That $2E\subset E\cap \extp^{2d}F$ is obvious, and we check the reverse inclusion. Let $Q$ be the quotient of $H^{2d}(A(\C)\times A'(\C),\Z)= \extp^{2d}(\Z e_1\oplus\dots\oplus\Z f'_d)$ by the subgroup generated by elements of the form 
$(e_1-e'_1)\wedge g_2\wedge\dots\wedge g_{2d}$  with $g_2,\dots, g_{2d}\in \{e_1,\dots, f'_d\}$ on the one hand, and by elements of the form $g_1\wedge \dots\wedge g_{2d}$ with $g_1,\dots, g_{2d}\in \{e_1,\dots, f'_d\}\setminus\{e_1,e'_1\}$ on the other hand.
Consider the $\Z/2$\nobreakdash-basis of $Q\otimes_{\Z}\Z/2$ consisting of the images of the elements of the form $e_1\wedge g_2\wedge\dots\wedge g_{2d}$ with $g_2,\dots, g_{2d}\in \{e_1,\dots, f'_d\}\setminus\{e_1,e'_1\}$.
Decompose in this basis the images in $Q\otimes_{\Z}\Z/2$ of the classes $\zeta_{k,l,m}$ considered in (ii).
For $k\geq l$, the basis element~(\ref{basiselement}) appears with nonzero coefficient only in the decomposition of the image of
$\zeta_{k,l,m}$, with one exception: when $(k,l,m)=(1,0,d-1)$, the basis element (\ref{basiselement}) appears with nonzero coefficient only in the decomposition of the images of
$\zeta_{1,0,d-1}$ and of $\zeta_{0,0,d}$. This shows that the classes $\zeta_{k,l,m}$ are $\Z/2$-linearly independent in $Q\otimes_{\Z}\Z/2$. As the image of $\extp^{2d}F$ in $Q\otimes_{\Z}\Z/2$ is zero,
we deduce that all the coefficients appearing in the decomposition of an element of $E\cap \extp^{2d}F$ in the $\Z$-basis of $E$ described in (ii) must be even. This shows that $E\cap \extp^{2d}F\subset 2E$, as wanted.
\end{proof}

\begin{lem}
\label{parity}
If $\delta,\delta'\in E$, then $\deg(\delta\cdot\delta')$ is even.
\end{lem}

\begin{proof}
It suffices to prove the lemma when $\delta$ and $\delta'$ belong to the $\Z$-basis of $E$ described in Lemma \ref{independence} (ii). We may thus assume that $\delta=\zeta_{k,l,m}$ and $\delta'=\zeta_{k',l',m'}$. If $k>l$, then 
$$\deg(\delta\cdot\delta')=\deg((\varepsilon_{k,l,m}+\varepsilon_{l,k,m})\cdot\delta')
=2\deg(\varepsilon_{k,l,m}\cdot\delta')$$ 
by $\Z/2$-invariance of $\delta'$, and this number is even. The same argument applies if $k'>l'$.
Assume now that $k=l$ and $k'=l'$. Then one has
\begin{equation}
\label{binomials}
\delta\cdot\delta'=\zeta_{k,k,m}\cdot\zeta_{k',k',m'}=\binom{k+k'}{k}^2\binom{m+m'}{m}\varepsilon_{k+k',k+k',m+m'}.
\end{equation}
As a consequence of \cite[Lemme 1 p.~247]{Beauville} and of the projection formula applied to the morphism $A\times A'\to A'$, one computes that
\begin{equation}
\label{Fourier}
\begin{alignedat}{2}
\deg(\varepsilon_{k+k',k+k',m+m'})&=(-1)^{d-k-k'}\deg_{A'}\big(\frac{(\lambda')^{k+k'}}{(k+k')\,!}\cdot\frac{(\lambda')^{d-k-k'}}{(d-k-k')\,!}\big)\\
&=(-1)^{d-k-k'}\binom{d}{k+k'}.
\end{alignedat}
\end{equation}
Combining (\ref{binomials}) and (\ref{Fourier}) yields $\deg(\delta\cdot\delta')=(-1)^{d-k-k'}\binom{k+k'}{k}^2\binom{m+m'}{m}\binom{d}{k+k'}$. 

Assume for contradiction that this number is odd.  Let us say that two integers $n, n'\geq 0$ are \textit{dyadically disjoint} if their dyadic expansions do not share any nonzero digit. 
The formula for the $2$-adic valuation of the factorial appearing in \cite[p.~241]{Robert} 
shows that  $\binom{n+n'}{n}$ is odd if and only if $n$ and $n'$ are dyadically disjoint.
 It follows that $k$ and $k'$ are dyadically disjoint, and that so are $m=d-2k$ and $m'=d-2k'$, and $k+k'$ and $d-k-k'$. As $d-2k$ and $d-2k'$ are dyadically disjoint, the integer $d$ is even. Write $d=2e$. Then $e-k$ and $e-k'$ are dyadically disjoint. As $k$ and $k'$, as well as $e-k$ and $e-k'$ and also $k+k'$ and $(e-k)+(e-k')$ are dyadically disjoint, the four integers $k$, $k'$, $e-k$ and $e-k'$ are pairwise dyadically disjoint. Since $k+(e-k)=k'+(e-k')$, this is impossible unless these four numbers all vanish. This contradicts the assumption that $d\geq 1$.
\end{proof}

 Let $G:=\Gal(\C/\R)$, and consider the $G$-module $\Z(j):=(\sqrt{-1})^j\Z\subset\C$. If $X$ is a variety over $\R$,  letting $G$ act both on $X(\C)$ and on $\Z(j)$ endows $H^k(X(\C),\Z(j))$ with an action of $G$. 

 \begin{prop}
 \label{resab}
 For all $d\geq 1$, there exists an abelian variety $X$ of dimension $2d$ over $\R$ and a class $\beta\in \CH^d(X)$ with the following properties.
 \begin{enumerate}[(i)]
 \item If $\gamma,\gamma'\in H^{2d}(X(\C),\Z(d))$ are Hodge and $G$-invariant, then $\deg(\gamma\cdot\gamma')$ is even.
 \item $\deg (\beta^2)\equiv 2\pmod 4$.
 \item $\cl_{\R}(\beta)=0\in H_d(X(\R),\Z/2)$.
 \end{enumerate}
 \end{prop}
 
 \begin{proof}
 Let $(A,\lambda)$ be a very general principally polarized abelian variety of dimension~$d$ which is defined over~$\R$ (by Baire's theorem, one may choose a very general real point of the moduli space $M$ of $d$-dimensional principally polarized abelian varieties with level $3$ structure, since $M$ is a smooth variety).  Write $x\mapsto \bar{x}$ for the action of complex conjugation on $A(\C)$.

Let $(A',\lambda')$ be another copy of $(A,\lambda)$.  The real abelian variety $A\times A'$ has set of complex points $A(\C)\times A'(\C)$ with an action of complex conjugation given by $(x,y)\mapsto (\overline{x},\overline{y})$.
Define ${X:=\Res_{\C/\R}(A)}$ to be the Weil restriction of scalars of $A_{\C}$. It is the abelian variety over $\R$ whose set of complex points is $X(\C)=A(\C)\times A'(\C)$, with an action of complex conjugation given by $(x,y)\mapsto (\overline{y},\overline{x})$.
The subvariety $A_{\C}\times\{0\}\cup\{0\}\times A'_{\C}$ of $A_{\C}\times A'_{\C}$ descends to a subvariety $Z\subset X$, and we define $\beta:=[Z]\in\CH^d(X)$. Since the normalization of~$Z$ has no real points, one has $\cl_{\R}(\beta)=0$. Moreover, $\deg(\beta^2)=2$. Assertions (ii) and~(iii) are proven.

Let $\mu\in H^2(A(\C)\times A'(\C),\Z(1))$ be the class of the Poincar\'e bundle. 
As the cycle class map $\CH^1(A_{\C}\times A'_{\C})\to H^{2}(A(\C)\times A'(\C),\Z(1))$ is $G$-equivariant, 
the classes $\lambda$, $\lambda'$ and $\mu$ all belong to $H^{2}(A(\C)\times A'(\C),\Z(1))^G$.  It then follows from Lemma \ref{independence} (i) that the group $\Hdg^{2d}(A(\C)\times A'(\C),\Z(d))$ consists exclusively of $G$-invariant classes. 

As the action of complex conjugation on $X(\C)$ is the composition of the action of complex conjugation on $A(\C)\times A'(\C)$ and of the exchange of the two factors,  we deduce that group of $G$-invariant Hodge classes in $H^{2d}(X(\C),\Z(d))$  is exactly the group $E$ described in Lemma \ref{independence} (ii). Assertion (i) now follows from Lemma~\ref{parity}.
\end{proof}
  
  The next proposition is a variant of Proposition \ref{resab} which works over the complex numbers, but which is slightly more complicated.

 \begin{prop}
 \label{quotab}
 For all $d\geq 1$, there exists a $2d$-dimensional smooth projective variety $X$ over $\C$ and a class $\beta\in \CH^d(X)$ with the following properties.
 \begin{enumerate}[(i)]
 \item If $\gamma,\gamma'\in H^{2d}(X(\C),\Z)$ are Hodge, then $\deg(\gamma\cdot\gamma')$ is even.
 \item One has $\deg (\beta^2)\equiv 2\pmod 4$.
\item All higher Chern classes of $X$ are torsion, i.e., $c(X)=1\in\CH^*(X)\otimes_{\Z}\Q$.
 \end{enumerate}
 \end{prop}
 
 \begin{proof}
 Let $(A,\lambda)$ be a very general principally polarized abelian variety of dimension~$d$ over~$\C$. Let $e_1,\dots,e_d,f_1,\dots,f_d\in H^1(A(\C),\Z)$ be a basis such that $\lambda=\sum_i e_i\wedge f_i$. Let $\tau\in A(\C)[2]\simeq H_1(A(\C),\Z/2)$ be the $2$-torsion point associated with the morphism $H^1(A(\C),\Z)\to\Z/2$ sending $e_1$ to $1$ and $e_2,\dots,e_d,f_1,\dots,f_d$ to $0$. Denote by $(A',\lambda'=\sum_ie'_i\wedge f'_i,\tau')$ another copy of $(A,\lambda=\sum_i e_i\wedge f_i,\tau)$. 
 
 Let $\Z/4$ act on $A\times A'$ via $(x,x')\mapsto (x'+\tau,x)$. Let $p:A\times A'\to X$ (resp. $q:A\times A'\to B$) be the quotient of $A\times A'$ by $\Z/4$ (resp. by the subgroup $\Z/2\subset \Z/4$). Since $\Z/2$ acts on $A\times A'$ via $(x,x')\mapsto(x+\tau,x'+\tau')$, we see that~$B$ is an abelian variety, and that $q^*(H^1(B(\C),\Z))\subset H^1(A(\C)\times A'(\C),\Z)$  is the subgroup generated by $e_2,\dots,e_d,f_1,\dots,f_d,e'_2,\dots,e'_d,f'_1,\dots,f'_d,2e_1, e_1+e'_1,2e'_1$.

Assertion (iii) follows at once from the fact that $c(A\times A')=1\in \CH^*(A\times A')$ since $p:A\times A'\to X$ is finite \'etale.

 Consider $Z:=p(A\times\{0\})\subset A\times A'$, and define $\beta:=[Z]\in\CH^d(X)$. As $p^*\beta=[A\times\{0\}]+[A\times\{\tau'\}]+[\{0\}\times A']+[\{\tau\}\times A']$, one computes that $\deg(p^*\beta^2)=8$. Since $\deg(p)=4$, one has $\deg(\beta^2)=2$. This proves assertion (ii).

Let $\gamma,\gamma'\in \Hdg^{2d}(X(\C),\Z)$ be Hodge classes. 
 As translations act trivially on the cohomology of $A\times A'$, the automorphism $(x,x')\mapsto (x'+\tau,x)$ acts in cohomology as $(x,x')\mapsto (x',x)$. It follows that $\delta:=p^*\gamma$ and $\delta':=p^*\gamma'$ are invariant under the involution of $\Hdg^{2d}(A(\C)\times A'(\C),\Z)$ exchanging the two factors. In other words, the classes~$\delta$ and $\delta'$ belong to the group denoted by $E$ in Lemma \ref{independence}. Since moreover~$\delta$ and~$\delta'$ belong to $\Ima(q^*)$,  Lemma \ref{independence} (iii) shows that $\delta$ and $\delta'$ are divisible by $2$ in $E$. Lemma \ref{parity} now implies that $\deg(\delta\cdot \delta')\equiv 0\pmod 8$. Since $\deg(p)=4$, we deduce that $\deg(\gamma\cdot \gamma')\equiv 0\pmod 2$, which proves~(i).
  \end{proof}

 \subsection{High-dimensional cycles}
 \label{high}

Here are applications to Chow groups.

\begin{thm}
\label{Chowth3}
Let $d\geq c$ be such that $\alpha(c+1)\geq 3$. Then there exists an abelian variety $X$ of dimension $c+d$ over $\R$ such that $\CH_d(X) $ is not generated by classes of closed subvarieties of $X$ that are smooth along their real loci.
 \end{thm} 

\begin{proof}
Suppose first that $d=c$, and let $X$ and $\beta$ be as in Proposition \ref{resab}. Assume for contradiction that $\beta=\sum_i n_i[Y_i]\in \CH_d(X)$ where $n_i\in\Z$ and the $Y_i$ are integral closed subvarieties of $X$ that are smooth along their real loci. One computes:
\begin{equation}
\label{congChow}
\deg(\beta^2)=\sum_i n_i^2\deg([Y_i]^2)+2\sum_{i<j}n_in_j\deg([Y_i]\cdot[Y_j]).
\end{equation} 
The existence of Krasnov's cycle class map $\cl:\CH_d(X)\to H^{2d}_G(X(\C),\Z(d))$ to $G$-equivariant Betti cohomology refining the usual complex cycle class map ${\cl_{\C}:\CH_d(X)\to H^{2d}(X(\C),\Z(d))}$ to Betti cohomology \cite[Theorem 0.6]{Krasnov2}, the fact that the image of $\cl_{\C}$ consists of Hodge classes, and assertion (i) of Proposition~\ref{resab}, combine to show that $2\sum_{i<j}n_in_j\deg([Y_i]\cdot[Y_j])$ is divisible by $4$.
Let $g_i:W_i\to Y_i$ be a resolution of singularities which is an isomorphism above~$Y_i(\R)$. Proposition \ref{FL} (ii) and the fact that all higher Chern classes of the abelian variety~$X$ vanish show that $\deg([Y_i]^2)\equiv \deg(s_d(W_i))\pmod 4$. Since $\alpha(d+1)\geq 3$ by hypothesis, Theorem \ref{RT} implies that $\deg(s_d(W_i))=s_d([W_i(\C)])\equiv 0\pmod 4$. These congruences and assertion (ii) of Proposition \ref{resab} contradict (\ref{congChow}).

To deal with the general case, apply the $d=c$ case to get a smooth projective variety $X'$ of dimension $2c$ over $\R$ and a class $\beta'\in\CH_c(X')$ that is not a linear combination of classes of subvarieties of $X'$ that are smooth along their real loci. Define $X:=X'\times A$ where $A$ is any abelian variety of dimension $d-c$ over $\R$, and $\beta:=pr_1^*\beta'\in\CH_d(X)$. That $\beta$ is not a linear combination of classes of subvarieties of $X$ that are smooth along their real loci follows from the corresponding property of $\beta'$ and from the Bertini theorem.
\end{proof} 

  \begin{thm}
 \label{Chowth4}
If $d\geq c$ are such that $\alpha(c+1)\geq 2$, there exists an abelian variety~$X$ of dimension $c+d$ over $\R$ such that $\Ker \big[\cl_{\R}:\CH_d(X)\to H_d(X(\R),\Z/2)\big]$  is not generated by classes of closed subvarieties of $X$ with empty real loci.
\end{thm}

\begin{proof}
The proof is almost identical to the proof of Theorem \ref{Chowth3}, replacing \textit{that are smooth along their real loci} by \textit{with empty real loci} everywhere. Only the argument used in the $d=c$ case to show that $s_d([W_i(\C)])\equiv 0\pmod 4$ needs to be modified, as follows. Since $W_i(\R)=\varnothing$, one has $\psi([W_i(\C)])=0\in MO_{d}$ by \cite[Theorem~22.4]{CF}. We deduce from Lemma \ref{cw} that all the Chern numbers of $[W_i(\C)]\in MU_{2d}$ are even. Theorem \ref{divtop} and the hypothesis that $\alpha(d+1)\geq 2$ imply that $s_d([W_i(\C)])\equiv 0\pmod 4$, as wanted.
\end{proof} 

  \begin{thm}
 \label{Chowth5}
If $d\geq c$ are such that $\alpha(c+1)\geq 3$, there exists a smooth projective variety~$X$ of dimension $c+d$ over $\C$ such that $\CH_d(X)$  is not generated by classes of smooth closed subvarieties of $X$.
\end{thm}

\begin{proof}
The proof is similar to that of Theorem \ref{Chowth3}.
The argument at the end of the proof of Theorem \ref{Chowth3} shows that we may assume that $d=c$. 

Let $X$ and $\beta$ be as in Proposition \ref{quotab}. Assume that $\beta=\sum_i n_i[Y_i]\in \CH_d(X)$ where $n_i\in\Z$ and the $Y_i$ are smooth closed subvarieties of $X$, and consider equation~(\ref{congChow}).
Since the Betti cohomology classes of the $Y_i$ are Hodge, assertion (i) of Proposition~\ref{quotab} shows that $2\sum_{i<j}n_in_j\deg([Y_i]\cdot[Y_j])$ is divisible by $4$. Assertion~(iii) of Proposition \ref{quotab} and \cite[Corollary 6.3]{Fulton} together imply that ${\deg([Y_i]^2)=\deg(s_d(W_i))}$. Since $\alpha(d+1)\geq 3$, Theorem \ref{RT} implies that $\deg(s_d(W_i))$ is divisible by $4$. Assertion~(ii) of Proposition \ref{quotab} now contradicts (\ref{congChow}).
\end{proof}

 \subsection{Hypersurfaces in abelian varieties}
 
 We give here a geometric construction based on a Noether--Lefschetz argument, on which the proof of Theorem \ref{thji} relies.
 
 \begin{prop}
 \label{constrHodge}
 For all $d,e\geq 1$, there exists a $2d$-dimensional smooth projective variety $X$ over $\R$ with the following properties.
 \begin{enumerate}[(i)]
 \item The total Chern class $c(X)\in\CH^*(X)$ of $X$ satisfies $c(X)\equiv 1\pmod {2^{e+1}}$.
 \item  The subgroup of Hodge classes $\Hdg^{2d}(X(\C),\Z) \subset H^{2d}(X(\C),\Z)$ is generated by a class $\eta\in \Hdg^{2d}(X(\C),\Z)$ with $\deg (\eta^2)\equiv 0\pmod {2^{e+1}}$.
 \item One has $X(\R)\neq\varnothing$.
 \end{enumerate}
 \end{prop}
 
 \begin{proof}
Arguing as in the proof of Proposition \ref{resab}, choose a very general principally polarized abelian variety $A$ of dimension $2d+1$ over $\R$.
The principal polarization of $A$ is represented by an ample line bundle $\cL$ on $A$ which is defined over $\R$ (see \cite[Theorem~4.1]{SS}). The group $\Hdg^{2d}(A(\C),\Z)$ of degree $2d$ Hodge classes on $A_{\C}$ is generated by $\frac{1}{d\,!}c_1(\cL)^{d}$ by Mattuck's theorem (see \cite[Theorem 17.4.1]{BL}) since $\frac{1}{d\,!}c_1(\cL)^{d}$ is a primitive integral cohomology class.

Let $l\gg 0$ be such that $2^{e+1}\,|\,l$ and $\cL^{\otimes l}$ is very ample. Choose a Lefschetz pencil of sections $\cL^{\otimes l}$, and let $X\subset A$ be a very general member of this pencil with $X(\R)\neq\varnothing$.

The restriction morphism $H^{2d}(A(\C),\Z)\to H^{2d}(X(\C),\Z)$
is injective with torsion free cokernel $\Lambda$, by the weak Lefschetz theorem \cite[Theorem 2]{AF}. Let $\Xi\subset \Lambda_{\C}$ be the subspace of Hodge classes. Since $X$ was chosen very general, any class in $\Xi$ remains Hodge when transported horizontally using the Gauss--Manin connection of the Lefschetz pencil. It follows that $\Xi$ is stabilized by the monodromy of the Lefschetz pencil on $\Lambda_{\C}$. Since this action is irreducible as a consequence of the hard Lefschetz theorem (see \cite[Theorem 7.3.2]{Lamotke}) and since the Hodge structure on~$\Lambda_{\C}$ is not trivial (as the restriction map $H^{2d}(A,\cO_A)\to H^{2d}(X,\cO_X)$ is not surjective), we deduce that $\Xi=0$. This shows that all Hodge classes in $H^{2d}(X(\C),\Z)$ are in the image of the injective restriction map $H^{2d}(A(\C),\Z)\to H^{2d}(X(\C),\Z)$, hence that $\Hdg^{2d}(X(\C),\Z)$ is generated by the restriction $\eta$ of $\frac{1}{d\,!}c_1(\cL)^{d}$. We now compute 
$\deg(\eta^2)=\deg(\frac{1}{(d\,!)^2}c_1(\cL)^{2d}\cdot c_1(\cL^{\otimes l}))=\frac{l(2d+1)!}{(d\,!)^2}$, which proves (ii).

 The normal exact sequence $0\to T_X\to (T_A)|_X\to \cL^{\otimes l}|_X \to 0$ finally shows that $c(X)=c(A)|_X\cdot c(\cL^{\otimes l}|_X)^{-1}=(1+lc_1(\cL)|_X)^{-1}\equiv 1\pmod{2^{e+1}}$, proving~(i).
 \end{proof}

\subsection{Submanifolds with no algebraic approximations}
\label{subnoapp}
 
Now come the prom\-ised applications to algebraic approximation.
 We refer to \S\ref{approximation} for the definition of the cobordism group $MO_d(X(\R))$ of the real locus of a smooth projective variety~$X$ over~$\R$, and of its subgroup $MO^{\alg}_d(X(\R))$ of algebraic cobordism classes. Recall that $MO_d$ denotes the cobordism ring, which is the cobordism group of the point (see \S\ref{cob}).
   
  \begin{thm}
  \label{thji}
  Assume that $c,d,e\geq 1$ are such that $d\geq c$ and $\alpha(c+e)=2e$. Then there exist a smooth projective variety $X$ of dimension $c+d$ over $\R$ and a $d$-dimensional closed $\ci$ submanifold $j:M\hookrightarrow X(\R)$ such that the following properties hold for all $d$-dimensional closed subvarieties $i:Y\hookrightarrow X$.
 \begin{enumerate}[(i)]
 \item One has $[j]\in MO^{\alg}_d(X(\R))$.
 \item  If $Y$ is smooth, then $[M]\neq[Y(\R)]\in MO_d$.
 \item If $e=1$ and $Y$ is smooth along $Y(\R)$, then $[M]\neq[Y(\R)]\in MO_d$.
   \end{enumerate}
\end{thm}
 
 We first prove Theorem \ref{thji} in the particular case where $c=d$.
 
 \begin{lem}
 \label{lemji}
If $c=d$, Theorem \ref{thji} holds with any $X$ as in Proposition \ref{constrHodge}.
 \end{lem}
 
 \begin{proof}
  Recall from \S\ref{swc} that if $I=(i_1,i_2,\dots)$ is a sequence of nonnegative integers, we set $|I|:=\sum_r ri_r$. By Theorem \ref{divtop}, there exists a degree $1$ homogeneous polynomial $P\in \Z[x_I]_{|I|=d}$ such that $s_d(y)\equiv 2^eP(c_I(y)) \pmod{2^{e+1}}$ for all $y\in MU_{2d}$. 
  
  Theorem \ref{RT} shows the existence of $y_0\in MU_{2d}$ such that $s_d(y_0)\not\equiv 0\pmod{2^{e+1}}$, hence such that $P(c_I(y_0))\not\equiv 0\pmod 2$. Letting $M$ be a compact $\ci$ manifold representing the $\psi(y_0)\in MO_d$, Lemma \ref{cw} shows that $P(w_I(M))\neq 0\in\Z/2$.
  
A theorem of Whitney \cite[Theorem 5]{Whitney}
 asserts that any $d$-dimensional compact $\ci$ manifold may be embedded in $\R^{2d}$.
It follows that one may choose an embedding $j:M\hookrightarrow X(\R)$ of $M$ in a small ball of $X(\R)$. Since this small ball is contractible, $j$ is homotopic (hence cobordant) to a constant map $M\to X(\R)$. As~$M$ is cobordant to the real locus of a smooth projective variety over $\R$ (see \S\ref{cob}), one deduces that ${[j]\in MO^{\alg}_d(X(\R))}$, proving (i).
  
 Let $W\to Y$ be a desingularization of $Y$ which is an isomorphism above $Y(\R)$ and let $g:W\to X$ be the induced morphism. Proposition \ref{FL} shows, under the assumptions of either (ii) or (iii), that
 \begin{equation}
 \label{congruence}
 \deg([Y]^2)=\deg ((g_*[W])^2)\equiv\deg(c_d(N_{W/X}))\pmod {2^{e+1}}.
 \end{equation}
  
  Since the Betti cohomology class of $Y$ is a Hodge class, assertion (ii) of Proposition~\ref{constrHodge} shows that $\deg([Y]^2)\equiv 0 \pmod{2^{e+1}}$. The total Chern class of $N_{W/X}$ is $c(N_{W/X})=c(g^*T_X)\cdot c(T_W)^{-1}=g^*c(X)\cdot s(W)$, where $s(W)\in \CH^*(W)$ denotes the total Segre class of $W$. We deduce from assertion (iii) of Proposition \ref{constrHodge} that
$c_d(N_{W/X})\equiv s_d(W)\pmod{2^{e+1}}$. Together with (\ref{congruence}), these facts shows that $s_d(W)\equiv 0\pmod{2^{e+1}}$.
  
  Consequently, we have $s_d(W(\C))\equiv 0\pmod{2^{e+1}}$. By our choice of $P$, we deduce that $P(c_I(W(\C)))\equiv 0\pmod 2$. Conner and Floyd \cite[Theorem 22.4]{CF} have proven that $W(\C)$ and $W(\R)\times W(\R)$ are cobordant, and it follows that ${\psi([W(\C)])=[W(\R)]\in MO_d}$. \!Lemma \ref{cw} now shows that ${P(w_I(W(\R)))=0\in \Z/2}$, hence that $P(w_I(Y(\R)))=P(w_I(W(\R)))\neq P(w_I(M))$. Since Stiefel--Whitney numbers are cobordism invariants (see \S\ref{swc}), we have $[M]\neq [Y(\R)]\in MO_d$.
 \end{proof}
 
 The proof of Theorem \ref{thji} in general easily reduces to the above lemma.
 
 \begin{proof}[Proof of Theorem \ref{thji}]
Lemma \ref{lemji} produces a smooth projective variety~$X'$ of dimension $2c$ over $\R$ and a $c$-dimensional closed $\ci$ submanifold ${j':M'\hookrightarrow X'(\R)}$ satisfying properties (i)-(iii) of Theorem \ref{thji} (with $d$ replaced by $c$).
Define $X:=X'\times\P^{d-c}_{\R}$ and $M:=M'\times\P^{d-c}(\R)$, and consider the embedding $j:=(j',\Id):M\hookrightarrow X(\R)$.

Our choice of $X'$ and $M'$, shows the existence of a morphism $g':W'\to X'$ of smooth projective varieties over $\R$ such that $j'$ is cobordant to $g'(\R)$. Defining $W:=W'\times\P^{d-c}_{\R}$ and $g:=(g',\Id)$, we see that $j$ is cobordant to $g(\R)$, hence that $[j]\in MO^{\alg}_d(X(\R))$, which proves (i).

Suppose now that $i:Y\hookrightarrow X$ is as in the statement of Theorem \ref{thji} and satisfies the hypothesis of either (ii) or (iii). Let $x\in\P^{d-c}(\R)$ be a general point, and define $Y':=Y\cap (X'\times\{x\})$ and $i':Y'\hookrightarrow X'\times\{x\}\simeq X'$ to be the natural inclusion. Bertini's theorem ensures that $Y'$ is smooth in case (ii) and that $Y'$ is smooth along $Y'(\R)$ in case (iii). If $M$ were cobordant to $Y(\R)$, Sard's theorem would imply that $M'$ is cobordant to $Y'(\R)$. This contradicts our choice of $M'$ and proves (ii) and~(iii).
 \end{proof}
  
\begin{rems}
\label{remNoether}
(i) 
Assertions (i) and (iii) of Theorem \ref{thji} prove Theorem \ref{thC2}.

(ii) 
It is striking that the obstructions to $M$ being approximable by real loci of algebraic subvarieties of $X$ provided by Theorem \ref{thji} (ii)-(iii) involve cobordism theory, although $[j]\in MO^{\alg}_d(X(\R))$ by Theorem \ref{thji} (i). Loosely speaking, the map $j$ is cobordant to an algebraic map, but not to an algebraic embedding.

(iii)
Complex cobordism and Theorem \ref{divtop} are not needed to prove Theorem~\ref{thji} for $c=2$. One may use Noether's formula instead, as in Example~\ref{Noether}.

(iv) In the setting of Theorem \ref{thji}, one could moreover arrange that the inclusion $j:M\hookrightarrow X(\R)$ be approximable in the $\ci$ topology by the inclusion of the set of smooth real points of an algebraic subvariety $Z\subset X$ with compact set of smooth real points (which necessarily also has some singular real points if $e=1$, in view of Theorem \ref{thji} (iii)). To prove it when $c=d$, one can use linkage and general position arguments as in Sections \ref{linkage} and \ref{sectionapprox}. We do not give a detailed proof here. The general case reduces to the case $c=d$ as in the proof of Theorem \ref{thji}.
\end{rems}



 The proof of Theorem \ref{projth} is a variant of the proof of Lemma \ref{lemji}. Since $c(\P^1_{\R}\times\P^{2^{k+1}-1}_{\R})\not\equiv 1\pmod 4$, the argument is slightly more complicated.
 
The relation $(\sum_rw_r(M))(\sum_r\ow_r(M))=1$ defines the \textit{normal Stiefel--Whitney classes} $\ow_r(M)\in H^r(M,\Z/2)$ of a compact $\ci$ manifold $M$. We also recall that if $i:Y\to X$ is a morphism of varieties over $\R$, we let $i(\R):Y(\R)\to X(\R)$ be the induced map between sets of real points.

 \begin{thm}
\label{projth}
Fix $k\geq 1$ and define $X:=\P^1_{\R}\times\P^{2^{k+1}-1}_{\R}$.
There exists a $2^k$\nobreakdash-dimen\-sional closed $\ci$ submanifold $j: M\hookrightarrow X(\R)$ such that $[j]\neq [i(\R)]\in MO_{2^k}(X(\R))$ for all $2^k$-dimensional closed subvarieties $i:Y\hookrightarrow X$ that are smooth along $Y(\R)$.
 \end{thm}

\begin{proof}
Since $\alpha(2^{k}+1)=2$, Theorem \ref{divtop}, shows that there is a degree $1$ homogeneous polynomial $P\in \Z[x_I]_{|I|=2^k}$ with $s_{2^k}(y)\equiv 2P(c_I(y)) \pmod{4}$ for all $y\in MU_{2^{k+1}}$. By Theorem~\ref{RT}, one may find $y_0\in MU_{2^{k+1}}$ such that $s_{2^k}(y_0)\not\equiv 0\pmod{4}$, hence such that $P(c_I(y_0))\not\equiv 0\pmod 2$. Letting $M$ be a compact $\ci$ manifold representing $\psi(y_0)\in MO_{2^k}$, Lemma \ref{cw} shows that $P(w_I(M))\neq 0\in\Z/2$.
  By Whitney's theorem \cite[Theorem 5]{Whitney}, one may embed $j:M\hookrightarrow X(\R)$ in a small ball of~$X(\R)$.

Let $i:Y\hookrightarrow X$ be as in the statement of Theorem \ref{projth}. Assume for contradiction that $[j]=[i(\R)]\in MO_{2^k}(X(\R))$. 
Let $W\to Y$ be a desingularization of $Y$ which is an isomorphism above $Y(\R)$ and let $g:W\to X$ be the induced morphism. 
Proposition~\ref{FL} (ii) shows that $\deg([Y]^2)=\deg ((g_*[W])^2)\equiv\deg(c_{2^k}(N_{W/X}))\pmod {4}$, hence that
\begin{equation}
\label{congruence2}
 \deg([Y]^2)\equiv\sum_{r=0}^{2^k} \deg(g^*c_r(X)\cdot s_{2^k-r}(W))\pmod {4}.
\end{equation}

  Consider the Borel--Haefliger cycle class map $\cl_{\R}:\CH^*(X)\to H^*(X(\R),\Z/2)$ (\cite{BH}, see also \cite[\S 1.6.2]{BW1}). Since $[j]=[i(\R)]\in MO_{2^k}(X(\R))$, one has $\cl_{\R}([Y])=[M]=0\in H^{2^k\!\!}(X(\R),\Z/2)$. Set $H_1:=c_1(\cO_{\P^1_{\R}}(1))\in \CH^1(\P^1_{\R})$ and $H_2:=c_1(\cO_{\P_{\R}^{2^{k+1}-1}}(1))\in \CH^1(\P^{2^{k+1}-1}_{\R})$. As $\CH^{2^k\!\!}(X)$ is generated by $(H_2)^{2^k}$ and $H_1(H_2)^{2^k-1}$, we compute that the kernel of $\cl_{\R}:\CH^{2^k\!\!}(X)\to H^{2^k\!\!}(X(\R),\Z/2)$ is generated by $2(H_2)^{2^k}$ and $2H_1(H_2)^{2^k-1}$. As a consequence, $[Y]\in \CH^{2^k\!\!}(X)$ is a multiple of $2$, and hence $ \deg([Y]^2)$ is divisible by $4$.

The Euler exact sequences $0\to \cO_{\P^N_{\R}}\to  \cO_{\P^N_{\R}}(1)^{\oplus N+1}\to T_{\P^N_{\R}}\to 0$ and the Whitney sum formula yield $c(X)=(1+H_1)^2(1+H_2)^{2^k}\in CH^*(X)$. Since ${H_1^2=H_2^{2^k}=0}$, we deduce that $c(X)\equiv 1\pmod 2$. 
For $r\geq 1$, let $\gamma_r\in \CH^r(X)$ be such that $c_r(X)=2\gamma_r$. Since Borel and Haefliger have shown that $\cl_{\R}(c(W))=w(W(\R))$ (\cite[\S 5.18]{BH}, see also \cite[Proposition 3.5.1]{Krasnov1}),
we have $\cl_{\R}(s(W))=\ow(W(\R))$. We deduce that, for $r\geq 1$,
\begin{alignat}{4}
\label{degmod2}
\deg(\cl_{\R}(g^*\gamma_r\cdot s_{2^k-r}(W)))&=\deg(g(\R)^*\cl_{\R}(\gamma_r)\cdot \ow_{2^k-r}(W(\R)))\nonumber\\
&=\deg(j^*\cl_{\R}(\gamma_r)\cdot\ow_{2^k-r}(M))\\
&=0\in\Z/2,\nonumber
\end{alignat}
where the first equality follows from the functorial properties of $\cl_{\R}$ (see \cite[\S 1.6.2]{BW1}), the second from the equality of the Stiefel--Whitney numbers of the cobordant maps $j$ and $g(\R)=i(\R)$ (see \cite[Theorem 17.3]{CF}), and the third holds since $j^*:H^r(X(\R),\Z/2)\to H^r(M,\Z/2)$ vanishes for $r\geq 1$ because the image of~$j$ is included in a small ball of $X(\R)$.

Equation (\ref{degmod2}) demonstrates that $\deg(g^*\gamma_r\cdot s_{2^k-r}(W))\in\Z$ is even, and hence that $\deg(g^*c_r(X)\cdot s_{2^k-r}(W))$ is divisible by $4$, for all $r\geq 1$. Plugging the congruences we have obtained into (\ref{congruence2}) shows that $\deg(s_{2^k}(W))\equiv 0\pmod 4$, hence that $s_{2^k}(W(\C))\equiv 0\pmod 4$. 
Our choice of $P$ implies that $P(c_I(W(\C)))\equiv 0\pmod 2$.

 By \cite[Theorem 22.4]{CF}, one has ${\psi([W(\C)])=[W(\R)]\in MO_{2^k}}$. Lemma \ref{cw} now shows that ${P(w_I(W(\R)))=0\in \Z/2}$, hence that $P(w_I(Y(\R)))\neq P(w_I(M))$. We deduce that ${[M]\neq [Y(\R)]\in MO_{2^k}}$ by cobordism invariance of Stiefel--Whitney numbers. A fortiori, $[j]\neq [i(\R)]\in MO_{2^k}(X(\R))$, which is a contradiction.
\end{proof}
  
 \begin{rems}
 \label{remP}
 (i)  Theorem \ref{projth} implies Theorem \ref{thP} by \cite[Proposition 4.4.4]{Wall}.

 (ii) Theorem \ref{projth} is false for $k=0$ by \cite[Theorem 12.4.11]{BCR}. The proof fails in this case because $\alpha(2^k+1)<2$ precisely for this value of $k$.
 
 (iii) The simplest particular case of Theorem \ref{projth} is the following. Embed $\P^2(\R)$ in $\R^4=\R\times\R^3$ and let $j:\P^2(\R)\hookrightarrow\P^1(\R)\times \P^3(\R)$ be the induced embedding.
 Then $j$ is not cobordant to the inclusion of the real locus of a closed subvariety of $\P^1_{\R}\times\P^3_{\R}$ which is smooth along its real locus. A fortiori, $j$ may not be isotoped to such a real locus. As in Remark \ref{remNoether} (iii), the use of Theorem \ref{divtop} may be replaced by Noether's formula in the proof of this particular case.

(iv) The conclusion of Theorem \ref{projth}, is not explained by a difference between the groups $MO_*^{\alg}(X(\R))$ and $MO_*(X(\R))$, as they coincide by \cite[Remark 3 p.~103]{BTvb} 
or \cite[Corollary~1~p.~314]{IS}.
 \end{rems}

\section{Algebraic approximation and algebraic homology}
\label{final}
  
In this section, we fix a smooth projective variety $X$ of dimension $c+d$ over~$\R$, and a closed $d$\nobreakdash-dimensional~$\ci$ submanifold $j:M\hookrightarrow X(\R)$. Recall from \S\ref{approximation} the definition of the approximation property $(A)$ and of its necessary condition~$(B)$ based on cobordism. We now study variants of these properties, our goal being Theorem \ref{thmfinal}.

\subsection{A stronger approximation property}

It is natural to consider the following strenghtening of the approximation property $(A)$ considered in \S\ref{approximation}:
\begin{property*}[$A'$]
For all neighbourhoods $\mathcal{U}\subset\ci(M,X(\R))$ of the inclusion, there exist ${\phi\in\cU}$ and a smooth closed subvariety $Y\subset X$ such that $\phi(M)=Y(\R)$.
\end{property*}

The following two theorems are analogues of Theorem \ref{thC1} and \ref{thC2} for the property~$(A')$. They are consequences of Theorems \ref{approxth} and \ref{thji} respectively.

\begin{thm}
\label{thC1bis}
Properties $(A')$ and $(B)$ are equivalent if $d<c$ and $d\leq 3$.
\end{thm}

 \begin{thm}
 \label{thC2bis}
Let $d\geq c$ and $e\geq 1$ be such that $\alpha(c+e)=2e$. Then there exist~$X$ and $M$ such that $(A')$ fails but $(B)$ holds.
 \end{thm}

\begin{rems}
(i)
We do not know if it is possible to get rid of the hypothesis that $d\leq 3$ in Theorem \ref{thC1bis}.

(ii)
Examples where $(B)$ (and even $(A)$) holds but $(A')$ fails had already been obtained by Akbulut and King \cite[Theorem 4]{AKtransc}, and refined by Kucharz \cite[Theorem 1.1]{Kuctransc}.
 Their examples work for all $(c,d)$ with $c\geq 2$ and $d\geq c+2$. 
The range of pairs $(c,d)$ that we reach in Theorem \ref{thC2bis} is different. 
\end{rems}

\subsection{Homology and cobordism obstructions}

Define $H_d^{\alg}(X(\R),\Z/2)$ to be the image of the Borel--Haefliger cycle class map $\cl_{\R}:\CH_d(X)\to H_d(X(\R,\Z/2)$.
Consider the following property.
\begin{property*}[H]
One has $j_*[M]\in H_d^{\alg}(X(\R),\Z/2)$.
\end{property*}
Property $(B)$ implies $(H)$ since algebraic homology classes are preserved by push-forwards (see for instance \cite[\S 1.6.2]{BW1}).
 It follows that $(H)$ is an obstruction to the approximation properties $(A)$ and $(A')$, that is weaker than $(B)$. In fact, the cobordism obstruction $(B)$ was first used by Bochnak and Kucharz \cite[Corollary~1.3]{BKsub} to give examples where $(A)$ fails but $(H)$ holds, when $d\geq 3$ and~$c\geq 2$.

\begin{lem}
\label{HB}
If $d\leq 2$, then properties $(B)$ and $(H)$ are equivalent.
\end{lem}

\begin{proof}
 Since $MO_1=0$, an isomorphism
$$(H_{d-2}(X(\R),\Z/2)\otimes MO_2)\oplus H_d(X(\R),\Z/2)\isoto MO_d(X(\R))$$
is constructed in \cite[Theorem 17.2]{CF}. It restricts to an isomorphism
$$(H_{d-2}(X(\R),\Z/2)\otimes MO_2)\oplus H_d^{\alg}(X(\R),\Z/2)\isoto MO_d^{\alg}(X(\R))$$
by a theorem of Ischebeck and Sch\"ulting \cite[Corollary 1 p.~314]{IS}, in view of the equality $H_0^{\alg}(X(\R),\Z/2)=H_0(X(\R),\Z/2)$. We deduce that the two conditions $j_*[M]\in H_d^{\alg}(X(\R),\Z/2)$ and  $[j:M\hookrightarrow X(\R)]\in MO_d^{\alg}(X(\R))$ are equivalent. 
\end{proof}

\subsection{Hypersurfaces} 
  
The following proposition is a well-known improvement of \cite[Theorem 12.4.11]{BCR}, which goes back to the work of Benedetti and Tognoli \cite[Proposition 1 p.~227]{BTapprox}
(see also \cite[Theorem A]{Akak}).
  
\begin{prop}
\label{hyp}
If $c=1$, then properties $(A')$ and $(H)$ are equivalent. 
\end{prop}
  
  \begin{proof}
Assume that $(H)$ holds.
Let $\mathcal{U}\subset\ci(M,X(\R))$ be a neighbourhood of the inclusion.
  By \cite[Theorem 12.4.11]{BCR}, there exist $\psi\in\cU$, an open neighbourhood~$U$ of $X(\R)$ in $X$ and a smooth closed hypersurface $Z\subset U$ with $\psi(M)=Z(\R)$. Let $\overline{Z}\subset X$ be the Zariski closure of $Z$. Since $X$ is smooth, there exist a line bundle~$\cL$ on $X$ and a section $s\in H^0(X,\cL)$ with $\oZ=\{s=0\}$.
  
    Fix a very ample line bundle $\cO_X(1)$ on $X$. Let $(u_1,\dots,u_N)$ be a basis of $H^0(X,\cO_X(1))$. The section $v:=\sum_{m=1}^N u_m^2\in H^0(X,\cO_X(2))$ 
vanishes nowhere on~$X(\R)$. Choose $l\gg0$ with $\cM:=\cL(2l)$ very ample, let $t\in H^0(X,\cM)$ be a general small deformation of $sv^l$, and set $Y:=\{t=0\}$, which is smooth by Bertini.
    
    That $Y$ has the required properties follows from \cite[\S 20]{AR}. More precisely, the proofs of \cite[Lemmas 20.3 and 20.4]{AR} applied with $X=X(\R)$, $Y=\cM(\R)$, $W\subset Y$ the zero section, $r\geq 1$, and $\mathcal{A}=\mathcal{C}^{r+1}(X(\R),\cM(\R))$ show that if $t$ is close to $sv^l$, then the inclusions ${Z(\R)\subset X(\R)}$ and $Y(\R)\subset X(\R)$ are isotopic, by an isotopy which is $\ci$ because so are $t$ and $sv^l$, and small in the $\ci$ topology (see the use of the implicit section theorem in the proof of \cite[Lemma 20.3]{AR}).
\end{proof}

\subsection{A question of Bochnak and Kucharz} 
   
In \cite[pp.~685-686]{BKsub}, Bochnak and Kucharz ask for which values of $c$ and $d$ are $(A)$ and $(H)$ equivalent. We obtain a full answer to that question, and disprove the expectation raised in \cite[p.~686]{BKsub} that $(A)$ and $(H)$ are not equivalent for $d=2$ and~$c\geq 3$.

 \begin{thm}
\label{thmfinal}
 Properties $(H)$, $(B)$, $(A)$ and $(A')$ are all equivalent in the following cases: if $c\leq 1$, if $d\leq 1$, or if $d=2$ and $c\geq 3$.

For all other values of $c$ and $d$, there exist $X$ and $M$ satisfying $(H)$ but not $(A)$. 
\end{thm}

\begin{proof}
The theorem is trivial if $c=0$ or $d=0$. It follows from Proposition \ref{hyp} if~$c=1$. The cases with $d\geq 3$ and $c\geq 2$ are covered by Bochnak and Kucharz in \cite[Corollary 1.3]{BKsub}.
 If $d\leq 2$, then $(H)$ is equivalent to $(B)$ by Lemma \ref{HB}. The cases where $d=1$ and $c\geq 2$, or where $d=2$ and $c\geq 3$, now follow from Theorem~\ref{thC1bis}.
Finally, when $c=d=2$, one may apply Theorem \ref{thC2bis}
\end{proof}

 \bibliographystyle{myamsalpha}
\bibliography{approx}

\end{document}